\documentclass[a4paper,18pt]{article}

\usepackage{amsmath}
\usepackage[pdftex,colorlinks]{hyperref}
\usepackage{indentfirst}
\usepackage[english]{babel}
\usepackage{graphics}

\usepackage{amssymb}
\usepackage{dsfont}
\usepackage{amsfonts}
\usepackage{amsthm}
\usepackage[dvips]{graphicx}
\usepackage[latin1]{inputenc}
\usepackage{enumerate}
\usepackage{amsxtra}
\usepackage{amstext}
\usepackage{amssymb}
\usepackage{latexsym}

\newcommand{\N}{\mathbb N}

\newcommand{\C}{\mathbb C}

\newcommand{\ov}{\overline}
\newcommand{\R}{\mathbb R}

\newcommand{\E}{\mathbb E}
\newcommand{\Z}{\mathbb Z}
\newcommand{\p}{\mathbb P}
\newtheorem{prop}{Proposition}[section]
\newtheorem{corollaire}[prop]{Corollary}

\newtheorem{defi}[prop]{Definition}
\newtheorem{lemme}[prop]{Lemma}
\newtheorem{remarque}[prop]{remark}

\newtheorem{theo}[prop]{Theorem}

\title{Random subgroups of linear groups are free}

\author{Richard Aoun \footnote{Laboratoire de Math\'ematiques, B\^atiment 425, Universit\'e Paris Sud 11, 91405 Orsay-
FRANCE, \newline E-mail: richard.aoun@math.u-psud.fr}}

\date{}
\makeindex
\begin{document}

\maketitle
\textbf{Abstract:}  We show that on an arbitrary finitely generated non virtually solvable linear group, any two independent random walks will eventually generate a free subgroup. In fact, this will hold for an exponential number of independent random walks.

\tableofcontents
\section{Introduction}

The Tits alternative \cite{tits} says that every finitely generated linear group which is not virtually solvable  contains a free group on two generators. A question that arises immediately is to see if this property is ``generic'' in the sense that two    ``random'' elements (in a suitable sense) on such groups  generate or not a free subgroup.  In recent works of Rivin - \cite{Rivin} - and Kowalski - \cite{Kowalski}-   where groups coming from an arithmetic setting are considered, similar situations occur: a random element is shown to verify a property $P$ with high probability, for example, a
random  matrix in one of the classical groups
$GL(n,\Z)$, $SL(n,\Z)$ or $Sp(n,\Z)$  has irreducible
characteristic polynomial. In our case we take two elements at random and the property $P$ will be `` generate a free subgroup ''. The method of the authors cited above relies deeply on arithmetic sieving techniques. In this paper, we consider an arbitrary finitely generated linear group, that is a subgroup of $GL_n(K)$ for some field $K$, and we
 use an entirely different set of techniques, namely random matrix products theory.

 Let us explain what we mean by choosing two elements ``at random'': a random element will be the realization of the random walk associated to some probability measure on the group. Formally speaking,  if $\mu$ is a probability measure on a group $\Gamma$, we denote by $\Gamma_\mu$ the smallest semigroup containing the support of $\mu$; we consider a sequence $\{X_n;n\geq 0\}$ of independent random variables on $\Gamma$ with the same law $\mu$, defined on a probability space $(\Omega, \mathcal{F},\p)$. The $n^{th}$ step of the random walk $M_n$ is defined by $M_n=X_1...X_n$. We will also consider
the reversed random walk: $S_n=X_n...X_1$.
\\ The main purpose of  this paper is to show the following statement, which answers a question of Guivarc'h \cite{Guivarch3} - 2.10-:

\begin{theo}\label{main}
Let $K$ be a field,  $V$ a finite dimensional vector space over $K$, $\Gamma$ a finitely generated non virtually solvable subgroup of $GL(V)$ equipped with two  probability measures $\mu$ and $\mu'$  having an exponential moment and such that $\Gamma_\mu=\Gamma_{\mu'}=\Gamma$. Let $(M_n)_{n\in \N^*}$, $(M'_n)_{n \in \N^*}$ be the independent random walks associated respectively to $\mu$ and $\mu'$. Then almost surely, for $n$ large enough, the subgroup $\langle M_n,M'_n \rangle$ generated by $M_n$ and $M'_n$ is free (non abelian).
More precisely,
\begin{equation} \label{eqini}\limsup_{n \rightarrow \infty}{ \frac{1}{n}\log\; \p\left(\langle M_n,M'_n \rangle \textrm{is not free} \right) }< 0 \end{equation}
\end{theo}
 These conditions are fulfilled when the support of $\mu$ (resp. $\mu'$) is a finite symmetric generating set, say $S$ (resp. $S'$) of $\Gamma$.  In this case,  $M_n$ (resp. $M'_n$ ) is a random walk on the  Cayley graph associated to $S$ (resp. $S'$). In other terms, if we consider the word metric, the theorem says that the probability that two  ``random'' elements in the
ball of radius $n$  do not generate a free subgroup is decreasing exponentially fast to zero; ``random'' here is to be
understood with respect to $n^{th}$ convolution power of $\mu$ (resp. $\mu'$). In this statement we could have taken $S_n$ instead of $M_n$.\\

 Let $\mu$ be a probability measure on $\Gamma$. For every integer $l$, we denote by $(M_{n,1})_{n\in \N^*}$,..., $(M_{n,l})_{n\in \N^*}$ a family of $l$ independent random walks associated to $\mu$. From the proof of Theorem \ref{main}, we will deduce the following stronger statement:
\begin{corollaire} \label{corexpo}There exists $C>0$ such that a.s., for all large $n$, $M_{n,1},...,M_{n,\lfloor exp(Cn) \rfloor}$ generate a free group on $l_n= \lfloor exp(Cn) \rfloor$ generators  \label{exocor}\end{corollaire}

\begin{remarque}  As explained above, our main result shares a common flavor with the works by Rivin and Kowalski   - \cite{Rivin} and \cite{Kowalski}-, in the sense that random elements in a finitely generated group are shown to verify a generic property with high probability.
Use of the theory of random matrix products allows us to treat  arbitrary finitely generated linear groups while  the arithmetic sieving techniques  in
 \cite{Rivin} and \cite{Kowalski}  use reduction modulo prime numbers and  deal with subgroups of arithmetic groups $G(\Z)$, where $G$ is an algebraic group. However what we loose is the effectiveness: in \cite{Rivin1}, Rivin proved that the bounds he obtains in \cite{Rivin} are   effective while ours are not. Indeed, our method uses the Guivarch-Raugi theorem on the separation of the first two Lyapunov exponents $\lambda_1$ and $\lambda_2$  and the known bounds on  $\lambda_1 -\lambda_2$ rely on the ergodic theorem and are thus non effective.  \end{remarque}

\begin{remarque}
In Guivarch's proof of the Tits alternative in \cite{Guivarch3} he showed that
$S_{n_k}$ et $S'_{n_k'}$ can be turned into ping-pong players (see Section \ref{generating}  for a definition of these terms) in a suitable linear
representation  for some subsequence $n_k$, $n_k'$ which were obtained as certain return times thanks to Poincar\'{e} recurrence. There is a substantial difficulty in
passing from some subsequence to the version we give in our main theorem. This situation
is not dissimilar to the difficulty encountered in \cite{densefree} where ping-pong players were
gotten from a precise control of the KAK decomposition, in contrast with Tits' original
argument which exhibited ping-pong players as high powers of proximal elements. \end{remarque}

In the proof, we will use the theory of random matrix products over an arbitrary local field (i.e. $\R$, $\C$, a $p$-adic field, or a field of Laurent series over a finite field). Very little literature exists on this topic apart from the case of real or complex matrices (\cite{guimier}). So, in this paper, we will develop most of the theory from scratch in the context of local fields.  Some of our statements will be just an adaptation of results known over the reals to arbitrary local fields  while some  are new even over $\R$. This is the case for Theorem \ref{expokak} which shows the exponential convergence of the K-components of the KAK decomposition,  and for Theorems \ref{independence} and \ref{independence1}, which prove the asymptotic independence of the directional components of the KAK decomposition. Similar statements for the Iwasawa decomposition can be found in \cite{Guivarch3}.  We refer the reader to Section \ref{subsproba} for the statements of these results. Let us only state here one of them regarding the asymptotic independence in the KAK decomposition.
\begin{theo}[Asymptotic independence in KAK with exponential rate] Let $k$ be a local field, $\mathbf{G}$ a $k$-algebraic group assumed to be semi-simple and $k$-split, $(\rho,V)$ an irreducible $k$-rational representation of $\mathbf{G}$.   Consider  a probability measure $\mu$ on $G=\mathbf{G}(k)$ with an exponential moment (see Definition \ref{moshader}) such that $\Gamma_\mu$ is Zariski dense in $G$ and $\rho(\Gamma_\mu)$ is contracting.  Let $\{X_n;n\geq 1\}$ be independent random variables with the same law $\mu$, $S_n=X_n...X_1$ the associated random walk. Denote by $S_n=K_nA_nU_n$ a $KAK$ decomposition of $S_n$ in $G$ (see Section \ref{preliminaries}). Denote by $e_1 \in V$ (resp. $e_1^* \in V^*$) a highest weight vector for the action of $A$ on $V$ via $\rho$ (resp. $\rho^*$ the contragredient representation). Then the random variables $K_n[e_1]$ and $U_n^{-1}.[e_1^*]$  are asymptotically independent in the following sense. There exist independent random variables $Z$ and $T$ on  $P(V)$ (resp. $P(V^*)$) with law the unique $\mu$-invariant (resp. $\mu^{-1}$-invariant) probability measure on $P(V)$ (resp. $P(V^*)$) such that the following holds. For every $\epsilon>0$,
there is some $\rho=\rho(\epsilon) \in ]0,1[$ such that for every $\epsilon$-Holder function $\phi$ on $P(V)\times P(V^*)$ and all large enough $n$, we have:

$$\big|\E \left( \phi (K_n[e_1],U_n^{-1}.[e_1^*] ) \right) - \E \left( \phi(Z,T) \right) \big| \leq \rho^n ||\phi||_\epsilon$$\label{marie-therese}\end{theo}
Here we have used the following notation: $V^*$ is the dual space of $V$, $P(V)$ (resp. $P(V^*)$) is the projective space of $V$ (resp. $V^*$) and $G$ acts on $V^*$ by the formula: $g.f(x)=f(g^{-1}x)$ for every $g\in G$, $f\in V^*$, $x\in V$. We have denoted by $\mu^{-1}$ the law of $X_1^{-1}$ and by $||\phi||_\epsilon$  the Holder constant of $\phi$: $$||\phi||_\epsilon= Sup_{[x],[y],[x'],[y']}\;{\frac{\big|\phi([x],[x'])-\phi([y],[y'])\big|}{\delta^{\epsilon}([x],[y])+\delta^{\epsilon}([x'],[y'])}}$$
where $\delta$ is the standard angle metric (i.e. Fubini-Study metric) on $P(V)$ and $P(V^*)$. A similar statement for the KAK decomposition of $\rho(S_n)$ in $SL(V)$ (see section \ref{uff}) holds: in this case, $G$ need not be assumed Zariski connected any longer (see Theorem \ref{independence1}). Although we have not checked, it is likely that the above result
holds without assuming that the Zariski closure of $\Gamma_\mu$ is semi-simple and $k$-split, but assuming
instead proximality and strong irreducibility.
\subsection{Outline of the paper}

In Section \ref{strategy}, we split the proof of our main theorem, i.e. Theorem \ref{main}, into two parts: an arithmetic part (Theorem \ref{local}) and a probabilistic part (Theorem \ref{main1}). In our work, the probabilistic part replaces the dynamical part of the original proof of the Tits alternative.  The arithmetic one is a variant of a classical lemma of Tits
\cite[Lemma 4.1]{tits} proved by Margulis and Soifer  \cite{marglis}. The probabilistic one  will be shown in Section \ref{secproof} using the results of Section \ref{subsproba}.\\

In Section \ref{generating}, we recall a classical method, known as ping-pong, to show that a pair of linear automorphisms  generate a free group.\\

Section  \ref{subsproba}  is the core of the paper and  constitutes a self-contained treatment of the basics of random matrix
theory over local fields. It can be read independently of the rest of the paper. To our knowledge, apart from \cite{guimier},
this is the first time that this subject is treated over non-archimedean fields. Over $\R$ or $\C$,
  this theory is well developed, starting with Furstenberg and Kesten in the 60's
 and later the French school in the 70's and 80's: Bougerol, Le Page,   Raugi and in particular Guivarc'h, whose work especially in   \cite{Guivarch3} and \cite{Guivarch} inspired us a lot.\\

One of our main goals in this section is to give limit theorems for the random walk $M_n$
in three
aspects: its norm, its action on projective space and its components in the Cartan
decomposition. Our main results
in this section are the following:
\begin{itemize}

\item Theorem \ref{direction} shows the exponential convergence in direction of the random walk
$M_n$. Namely,
under the usual assumptions, for every point $[x]$ on the projective space, $M_n[x]$
converges
exponentially fast to a random variable $Z$ on the projective space.

\item Theorem \ref{hausdweak} and more precisely its proof shows the exponential decay of the probability that $M_n[x]$ lies in a
given hyperplane,
uniformly over the hyperplane. We deduce that the unique $\mu$-invariant measure has some regularity.

\item Theorem \ref{expokak} shows that the $K$-components of the random walk $M_n$ in the Cartan
decomposition converge
exponentially fast.

\item Theorem \ref{independence} proves that the $K$-components of the random walk $M_n$ in the Cartan
decomposition become
independent asymptotically.

\end{itemize}

 Theorem \ref{hausdweak} is a weaker version of a well-known statement over $\R$ or $\C$. Its proof can be found in
Bougerol's book and is due to Guivarc'h \cite[Theorem 7']{Guivarch3}.  We will
verify that it holds over an arbitrary local field. Theorems \ref{direction}, \ref{expokak} and \ref{independence} on the other
hand are new even over
$\R$ (on $\R$ or $\C$ only the exponential rate is new). They also hold over an arbitrary local field, and so does everything we do in
Secion \ref{subsconv}.
The analog of Theorem \ref{independence} for the orthogonal and unipotent parts of the Iwasawa
decomposition was proven over $\R$ by
Guivarch in \cite[Lemma  8]{Guivarch3}.

Our proof of Theorems \ref{hausdweak} is not an mere translation of the standard proof of
this statement over
the reals. Rather we take a different and more direct route via our key cocycle lemma,
Lemma \ref{cocycle}, a result giving control on the
growth of cocyles in an abstract context. This lemma is itself an extension of a result
of Le Page (see the proof of \cite[Theorem 1]{Page})
which was key in his proof of the spectral gap on Holder functions on projective space
(\cite[Proposition 4]{Page}).

Another key ingredient and intermediate step is our Proposition \ref{largedeviations}, which says that,
under the usual assumptions,
for every given non zero vector $x$, with high probability the ratio $||M_nx||/||M_n||$
is not too small. This fact
can be interpreted as a weak form of Le Page's large deviation theorem in $GL_n(\R)$.

Our proof of Theorem \ref{expokak} is based on this approach as well and makes key use of the
cocyle lemma, Lemma \ref{cocycle} and
of Proposition \ref{largedeviations}. Theorem \ref{direction} is also an important ingredient in the proof of \ref{expokak}.
Finally the proof of
Theorem \ref{independence} combines all of the above.

We note that two Cartan decompositions will be considered in Section \ref{subsproba}, the one coming
from the ambient $SL_d(k)$
and the one attached to the (semi-simple) algebraic group in which the group generated by the
random walk is Zariski
dense. Our limit theorems will be proved in the two cases. In fact the results for the
Cartan decomposition
in $SL_d(k)$, which are our main interest, will be deduced from the analogous results in
the algebraic group. These statements will be  deduced from a delicate study of the Iwasawa decomposition in the algebraic group  (Theorem \ref{iwa}).
If this Zariski closure is not Zariski connected, further technicalities arise. They will
be dealt with in Section
\ref{subsequi} using standard Markov chains and stopping times techniques.

Finally, we note that our proofs rely deeply on the pointwise ergodic theorem via our
cocycle lemma, Lemma \ref{cocycle}.\\

Section \ref{secproof} is devoted to the proof of Theorem \ref{main1}, i.e. the probabilistic part of our
main result, using the results of Section \ref{subsproba}.

\paragraph{Acknowledgments} I sincerely thank my supervisor Emmanuel Breuillard for pointing me out this question, for his great availability, his guidance through my Ph.D. thesis and many remarks on an anterior version of this paper. I'm also grateful to Yves Guivarc'h whose work inspires me a lot.
\section{Preliminary reductions}
\label{strategy}
In this section we reduce the proof of Theorem \ref{main} to its probabilistic part, i.e. Theorem \ref{main1} below.

\subsection{Notation and terminology}
All random variables will be defined on a probability space $(\Omega, \mathfrak{F},\p)$. $\E$ refers to the expectation with respect to $\p$. The symbol ``a.s.'' refers to almost surely. Let us recall the definition of a random walk on a group:
\begin{defi} [Random walks on groups]
 Let $\Gamma$ be a discrete group,  $\mu$  a probability measure on $\Gamma$, $(X_i)_{i\in \N^*}$ a family of independent random variables on $\Gamma$ with the same law $\mu$. For each $n$, we define the $n^{\textrm{th}}$ step of the following random walks by:
$$M_n=X_1...X_n\;\;\;;\;\;\;S_n=X_n...X_1$$
The product being the group law of $\Gamma$. We denote by $\Gamma_\mu$ the smallest semigroup containing the support of $\mu$.
 \end{defi}
\begin{remarque}
For our main Theorem \ref{main}, there will be no difference taking the natural $(M_n)$ or the reversed random walk $(S_n)$ as explained in the Remark \ref{reversed} below. Note however that the asymptotic behavior of the two walks is not the same in general.  \end{remarque}
 When $\Gamma$ is a finitely generated group, $\Gamma$ is a metric space for the word length distance:
 for each symmetric generating set $S$ containing $1$, define: $l_S(g)=Min\{r; g=s_1...s_r;\;s_i \in S \;\forall i=1,...,r\}$.\\
  The following  defines then a distance on $\Gamma$: $d_S(g,g')=l_S(g'^{-1}g)\;\; g,g'\in \Gamma$. 
 \begin{defi}[Exponential moment on finitely generated groups] Let $\mu$ be a probability measure on a finitely generated group $\Gamma$. Let $S$ be as above. We say that $\mu$ has an exponential moment if there exists $\tau>0$ such that:
 $$\int{exp\left(\tau l_S(g) \right) d\mu(g)} < \infty$$
 It is immediate that having exponential moment is independent of the choice of the generating set defining $l_S$. \label{expmomentt}\end{defi}

Let us recall our main result in this paper:
\paragraph{Theorem}
\emph{Let $K$ be a field, $V$ a finite dimensional vector space over $K$, $\Gamma$ a finitely generated non virtually solvable subgroup of $GL(V)$ equipped with two  probability measures $\mu$ and $\mu'$  having an exponential moment and such that $\Gamma_\mu=\Gamma_{\mu'}=\Gamma$. Let $(M_n)_{n\in \N^*}$, $(M'_n)_{n \in \N^*}$ be two independent random walk associated respectively to $\mu$ and $\mu'$. Then almost surely, for $n$ large enough, the group $\langle M_n,M'_n \rangle $ generated by $M_n$ and $M'_n$ is free (non abelian).
More precisely,
\begin{equation} \label{eqini}\limsup_{n \rightarrow \infty}{ \frac{1}{n}\log\; \p\left(\langle M_n,M'_n \rangle  \textrm{is not free} \right) }< 0 \end{equation}}\\

\begin{remarque} The assumptions on $\mu$ (resp. $\mu'$) of the theorem are clearly fulfilled if the support of $\mu$ (resp. $\mu'$) is a finite, symmetric generating set of $\Gamma$ \end{remarque}

\begin{remarque} The bound (\ref{eqini}) implies that there exists $\rho \in ]0,1[$ such that for $n$ large enough, \begin{equation}\p\left(\langle M_n,M'_n\rangle \textrm{is not free} \right) \leq \rho^n\label{zekrayat}\end{equation}
By the Borel-Cantelli lemma, it suffices to prove the first assertion of the theorem.
Hence in the rest of the paper, we will focus on showing (\ref{zekrayat}). \end{remarque}
\begin{remarque} \label{reversed} There is no difference taking $(M_n)_{n\in \N^*}$ or the reversed random walk in Theorem \ref{main}. In fact, the increments are independent and have the same law which implies that ${(X_1,...,X_n)}$ has the same law as ${(X_n,...,X_1)}$ for every integer $n$, hence (\ref{eqini}) is unchanged if we replaced $M_n$ by $S_n$.\end{remarque}

\subsection{Outline of the proof of Theorem \ref{main}}
 A local field (i.e. a commutative locally compact field) is isomorphic either to $\R$ or $\C$ (archimedean case) or a finite extension of the $p$-adic field $\mathbb{Q}_p$ for some prime $p$ in characteristic zero or to  the field of formal Laurent series $L((T ))$ over a finite field $L$. When $k$ is archimedean, we denote by $|.|$ the Euclidean absolute value. When $k$ is not archimedean, we denote by $\Omega_k$ its  discrete valuation ring, $\pi$ a generator of its unique maximal ideal, $q$ the degree of its residual field, $v(.)$ a discrete valuation and consider the following ultrametric norm: $|.|=q^{-v(.)}$. \\

When we consider a  finitely generated \textbf{linear group}  $\Gamma$, i.e. $\Gamma\subset GL_d(K)$ for some $d\geq 2$ and a finitely generated field $K$,  we can benefit from other nice metrics than the word metric:  for each local field $k$ containing $K$, $\Gamma$ can be considered as a metric space with the topology of $End_d(k)$  induced on $\Gamma$. This justifies the two parts of our proof:  the arithmetic part (Theorem \ref{local})  which consists in finding a suitable local field containing $K$  and the probabilistic one (Theorem \ref{main1})   consisting in using limit theorems for random walks on linear groups over local field.  Theorem \ref{local} will be borrowed from \cite{marglis} and Theorem \ref{main1} is the main part of this paper. Before stating them and showing how they provide a proof of Theorem \ref{main}, we give some basic definitions:
\begin{defi}\label{defdef} (Strong irreducibility and contraction properties)\\
$\bullet$ \textbf{Strong irreducibility }: let $K$ be a field, $V$ a vector space over $K$ and $\Gamma$ a  subgroup of $GL(V)$. The action of $\Gamma$ on $V$ is said to be strongly irreducible  if $\Gamma$ does not fix a finite union of proper subspaces of $V$. This is equivalent to saying that $\Gamma$ contains no subgroup of finite index that acts reducibly on $V$.  In particular, if the Zariski closure $\ov{\Gamma}$ is connected  then irreducibility and strong irreducibility are equivalent (because the identity component of $\ov{\Gamma}$ is contained in any algebraic subgroup of finite index - \cite{Humphreys}-). We note that this notion is ``algebraic'' in the sense that $\Gamma$ is strongly irreducible if and only if $\ov{\Gamma}$ is. \\
$\bullet$ \textbf{Contraction  for local fields}:
 Let $(k,|.|)$ be a local field, $V$ a  vector space over $k$ and $\Gamma$ a subgroup of $GL(V)$. We choose any norm $||.||$ on $End(V)$.  We say that a sequence ${(\gamma_n)}_{n\in \N}\subset \Gamma^{\N}$ is contracting, if ${r_n\gamma_n}$ converges,  via a subsequence, to a rank one endomorphism for every (or equivalently one) suitable normalization $(r_n)_{n\in \N}$ of $k$ such that $||r_n \gamma_n||=1$.  It is equivalent to say  that the projective transformation
$[\gamma_n] \in PGL(V)$ contracts  $P(V)$ into a point, outside a hyperplane. Note that in the  archimedean case,
this is just saying that $\frac{\gamma_n}{ ||\gamma_n||}$ converges to a rank one endomorphism. \\
A representation $\rho$ of $\Gamma$ is said to be contracting if the group $\rho(\Gamma)$ contains a contracting sequence.\end{defi}

The following classical lemma gives a more practical method to verify contraction. It will be useful to us in Section \ref{subsequi}.
\begin{lemme} [Contraction and proximality] An element $\gamma\in GL(V)$ is said to be proximal if and only if  it has a unique
eigenvalue of maximal modulus. If $\Gamma$ contains a proximal
element then it is contracting. If $\Gamma$ acts irreducibly on $V$ and is contracting then it contains a proximal element. \label{proximal}\end{lemme}
\begin{proof} If $\gamma\in \Gamma$ is proximal, then its maximal eigenvalue $\lambda$ belongs to the field $k$ and the corresponding eigendirection is defined on $k$. The latter has a $\gamma$-invariant supplementary hyperplane defined on $k$. Consequently, in a suitable basis, $\gamma$ is of the form: $\left(
                                               \begin{array}{cc}
                                                 \lambda & 0 \\
                                                 0 & M \\
                                               \end{array}
                                             \right)$. By the spectral radius formula, we deduce that sequence $\{\gamma^n;n\in \N\}$ is contracting.
Conversely, consider sequences $\{\gamma_n;n\in \N\}$ in $\Gamma$, $\{r_n;n\in \N\}$ in $k$ such that $r_n\gamma_n$ converges to a rank one endomorphism $h$. $h$ is proximal if and only if $Im(h)\not\subset Ker(h)$. Suppose first that $h$ is proximal and notice that $\{g\in End(V); \textrm{\;$g$ is proximal}\}$ is open (for the topology on $End(V)$ induced by that of the local field $k$);  hence for sufficient large $n$, $r_n\gamma_n$ is proximal, a fortiori $\gamma_n$ and we are done. If $h$ fails to be proximal, or equivalently $Im(h)\subset Ker(h)$, we claim that one can still find $g\in \Gamma$ such that $gh$ is proximal; this would end the proof since by the same reasoning $g\gamma_n$ would be proximal for large $n$. Let us prove the claim: denote by $k x_0$ the image of $h$ and notice that $V= Vect\{gx_0; g\in \Gamma\}$ because the action of $\Gamma$ on $V$ is irreducible. Consequently, there exists $g\in \Gamma$ such that $gx_0 \not \in Ker(h)$. But $gx_0=Im(gh)$ and $Ker(h)=Ker(gh)$; whence $gh$ is proximal.

\end{proof}

\begin{defi}[Exponential local moment on linear groups]\label{bnb}
Let $k$ be a local field, $d$ an integer $\geq 2$,  $\Gamma$ be a subgroup of $SL_d(k)$, $||.||$ a  norm on
$End_d(k)$, $\mu$ a probability measure on $\Gamma$. We say that $\mu$ has an exponential local moment if
for some $\tau>0$, $$\int {||g||^\tau d\mu(g)} < \infty $$
\end{defi}

\begin{remarque}[Interpretation]
The definition above can be reformulated as follows: there exists $\tau>0$ such that \;$\int{exp(\tau \log ||g|| )d\mu(g)}<\infty$ or equivalently  $\int{exp(\tau d_X(\ov{g},\ov{I_d}) )d\mu(g)}<\infty$ where $X=SL_d(k)/K$ is the symmetric space associated to $SL_d(k)$ (see Section \ref{subsconv} for definition of $K$), $d_X(g_1,g_2)= \log ||g_2^{-1}g_1||$ is a distance on $X$, $I_d$ is the identity matrix of order $d$. \end{remarque}

Now we are able to state the two results.
In the following theorem, for a measure $\mu$ on $SL_d(k)$,  $\Gamma_\mu$ denotes the smallest \textbf{closed} semigroup containing the support of $\mu$.
\begin{theo} [Probabilistic part]  \label{main1} Let $k$ be a local field, $d\geq 2$, $\mu$, $\mu'$ two probability measures on $SL_d(k)$ having an exponential local moment and such that $\Gamma_\mu=\Gamma_{\mu'}$ is a strongly irreducible and contracting \textbf{subgroup}. We assume its Zariski closure to be $k$-split  and its connected component semi-simple. We denote by $(M_n)_{n\in \N^*}$ (resp. $(M'_n)_{n\in \N^*}$) the random walks associated to $\mu$ (resp. $\mu'$). Then a.s. for all $n$ large enough, the group $\langle M_n,M'_n \rangle$ generated by $M_n$ and $M'_n$ is free. More precisely,
\begin{equation} \label{eqini1}\limsup_{n \rightarrow \infty}{ \frac{1}{n}\log\; \p\left(\langle M_n,M'_n \rangle \textrm{is not free} \right) }< 0 \end{equation}

\end{theo}
\begin{remarque}
The assumptions $\ov{\Gamma_\mu}$  semi-simple and $k$-split can be dropped: $\Gamma_\mu$ being strongly irreducible,
the Zariski connected component of $\ov{\Gamma_\mu}$ is immediately reductive and everything
 we will do in Section \ref{subsestimate} for semi-simple groups is applicable to reductive groups.
 The assumption $k$-split will be used to simplify the Cartan and Iwasawa decompositions in Sections
 \ref{subsestimate} and \ref{preliminaries}, however similar decompositions hold in the general case.
  To keep the exposition as simple as possible we kept these conditions.
\end{remarque}

If $V$ is a vector space over a field $k$ and $\Gamma$ a group, we say that a representation $\rho:\Gamma \longrightarrow GL(V)$ is absolutely (strongly) irreducible if it remains (strongly) irreducible on $V\otimes_k k'$ for every algebraic extension $k'$ of $k$.
\begin{theo}[Arithmetic part]\cite[Theorem 2]{marglis}
Let $K$ be a finitely generated field, $G$ an algebraic group over $K$ such that the Zariski connected component $G^0$ is not solvable, $\Gamma$ be a $K$-Zariski dense subgroup.  Then there exists a local field $k$ containing $K$, a vector space $V$ over $k$ and a $k$-algebraic absolutely strongly irreducible representation $\rho: G \longrightarrow SL(V)$ such that $\rho(\Gamma)$ is contracting and the Zariski component of $\rho(G)$ is a semi-simple group.
\label{local} \end{theo}
 \begin{remarque} A classical lemma of Tits -\cite{tits}- says (or at least implies) the same as Theorem \ref{local} except that $\rho$ is a representation of a finite index subgroup
 of $G$. This is insufficient for us because the random walk  lives in all of $\Gamma$. However, when $G$ is Zariski connected
the above theorem and the aforementioned lemma of Tits are exactly the same. We note that the proof of Theorem \ref{local}
by Margulis and Soifer depends heavily on the classification of  semi-simple algebraic groups through their Dynkin diagram.  A more conceptual proof can be found in \cite{topological} except that
 the representation $\rho$
 takes value in $PGL(V)$, and this is not enough for our purposes. \end{remarque}

\paragraph{End of the proof of Theorem \ref{main} modulo Theorem \ref{main1} }
Let $\Gamma=\Gamma_\mu=\Gamma_{\mu'}$.  Since $\Gamma$ is  finitely generated, we can replace $K$ with the field generated over its prime field by the matrix coefficients of the (finitely many) generators of $\Gamma$. Let $G$ be the Zariski closure of $\Gamma$.  Then, we can apply Theorem \ref{local}. It gives a local field $k$, a $k$-rational absolutely strongly irreducible representation $(\rho,V)$ of $G$ such that the Zariski-connected component of $H=\rho(G)$ is semi-simple and $\rho(\Gamma)$ is contracting. Passing to a finite extension of $k$ if necessary, $H$ can be assumed $k$-split; $\rho$ remains absolutely strongly irreducible. We are now in the situation of Theorem \ref{main1}: we have a probability measure $\rho(\mu)$ (image of $\mu$ under $\rho$) on some $SL_d(k)$ such that $\Gamma_{\rho(\mu)}=\rho(\Gamma)$ is strongly irreducible and contracting. Moreover, the connected component of its Zariski closure  $H$  is semi-simple and $k$-split. To apply Theorem \ref{main1} we only have to check that $\rho(\mu)$ has an exponential local moment knowing that $\mu$ has an exponential moment. Indeed, if $g=s_1^{n_1(g)}...s_r^{n_r(g)}\in Supp(\mu)$ is
  a minimal expression of $g$ in terms of the generators of a symmetric finite generating set $S$ of $\Gamma$, then $l_S(g)=|n_1(g)|+...+|n_r(g)|$ whence  $||\rho(g)|| \leq  \big[Max\{\log||\rho(s)|| \vee \log||\rho(s^{-1})||;s\in S\}\big]^ {l_S(g)}$.  \; Consequently, if $\E \left(exp(\tau l_S(X_1) ) \right)$ is finite, then for some $\tau'>0$, $\E \left( ||\rho(X_1)||^{\tau'}\right)$ is also finite.
   We can now apply Theorem \ref{main1}: a.s., for $n$ large enough, $\langle \rho(M_n),\rho(M'_n) \rangle$ is free, a fortiori $\langle M_n,M'_n \rangle $ is also free. This ends the proof.\begin{flushright} $\Box$ \end{flushright}

\section{Generating free subgroups in linear groups}
\label{generating}
In Theorem \ref{main1} we must show that $M_n$ and $M'_n$ generate a free group. Below we use the classical ping-pong method to obtain two generators of a free subgroup.
 For a detailed description of these ping-pong techniques one can refer  to \cite{densefree} for a self-contained exposition or to the original article of Tits \cite{tits}.
\subsection{The ping-pong method}
Let  $k$  be a local field, $V$ a vector space over $k$, $P(V)$ its projective space, $\delta$ the Fubini-Study distance on $P(V)$ defined by:
$$\delta([x],[y])=\frac{||x \wedge y ||}{||x|| ||y||}\;\;\;\;\;;\;\;\;\;\; [x],[y]\in P(V)$$
where $[x]$ is the projection of $x\in V \setminus\{0\}$ on $P(V)$.

 \begin{itemize}
\item Let $\epsilon \in ]0,1[$. A projective transformation $[g] \in PSL(V)$
is called \textbf{$\epsilon$-contracting} if there exists a point $v_g \in P(V)$, called an attracting
point of $[g]$, and a projective hyperplane $H_g$, called a repelling
hyperplane of $[g]$, such that $[g]$ maps the complement of the $\epsilon$-neighborhood
of $H_g \subset  P(V)$ into the $\epsilon$-ball around $v_g$. We say that $[g]$ is \textbf{$\epsilon$-very contracting} if both $[g]$ and $[g^{-1}]$ are $\epsilon$-contracting.
\item $[g]$ is called \textbf{$(r, \epsilon)$-
proximal} ($r > 2\epsilon > 0$) if it is $\epsilon$-contracting with respect to some attracting
point $v_g \in P(V)$ and some repelling hyperplane $H_g$, such that $\delta(v_g;H_g) > r$. The transformation $[g]$ is called \textbf{$(r, \epsilon)$-very proximal} if both $[g]$ and $[g]^{-1}$
are $(r, \epsilon)$-proximal.
\item A pair of projective transformations $a, b \in PSL(V)$ is called \textbf{a ping-pong pair} if both $a$ and $b$ are $(r, \epsilon)$-very proximal, with respect to some $r > 2\epsilon > 0$, and if the attracting points of $a$ and $a^{-1}$ (resp. of $b$ and $b^{-1}$) are at least $r$-apart from the repelling hyperplanes of $b$ and $b^{-1}$
(resp. of $a$ and $a^{-1}$).
More generally, a $m$-tuple of projective transformations $a_1,..., a_m$ is called
a ping-pong $m$-tuple if all $a_i$'s are $(r, \epsilon)$-very proximal (for some $r > 2\epsilon >
0$) and the attracting points of $a_i$ and $a_i^{-1}$
 are at least $r$-apart from the
repelling hyperplanes of $a_j$ and $a_j^{-1}$, for any $i\neq j$.\\
\end{itemize}

The following useful lemma is an easy exercise:
\begin{lemme} [Ping-pong lemma]
If $a,b \in PSL(V)$ form a ping-pong pair then the subgroup $\langle a,b \rangle$ generated by $a$ and $b$ is free.
More generally if $a_1,...,a_m$ is a ping-pong $m$-tuple then $\langle a_1,...,a_m \rangle $ is free.
 \label{ping-pong}\end{lemme}
\subsection{ The Cartan decomposition}
\label{uff}
Let $d\geq 2$, $V=k^d$ and $(e_1,...,e_d)$ its canonical basis.\\
The attracting points and repelling hyperplanes are not unique. In this article, they will be defined via the Cartan decomposition in $SL(V)$. Let's recall it. \\\smallskip

When $k=\R$ or $\C$, consider the usual Euclidean (resp. Hermitian) norm on $k^d$ and the canonical basis $(e_1,...,e_d)$. Let $K=SO_d(k)$ (resp. $SU_n(\C)$ ) be the orthogonal (resp. unitary) group, $A=\{diag(a_1,...,a_d);\;a_i>0\;\forall i=1,...,d;\;\prod_{i=1}^d a_i=1\}$,
$A^+=\{diag(a_1,...,a_d)\in A; a_1\geq...\geq a_d >0\}$. In this setting, the Cartan decomposition holds: $SL_d(k)=KA^+K$. This is the classical polar decomposition.\\

 When $k$ is non archimedean, denote $K=SL_d(\Omega_k)$ and $A=\{diag(\pi^{{n_1}},...,\pi^{{n_d}});\; n_i\in \Z \;\forall i=1,...,d; \sum_{i=1}^d{n_i}=0\}$;
$A^+=\{diag(\pi^{{n_1}},...,\pi^{{n_d}})\in A; \;n_1\leq...\leq n_d\}$. If we consider the Max norm on $V$: $||x||=Max\{|x_i|; i=1,...,d\}$, \; $x\in V$, then one can show that $K$ is the group of isometries of $V$. With these notations, the Cartan decomposition is: $SL_d(k)=KA^+K$.  This decomposition can be seen as an application of the well-known Invariant Factor Theorem for Matrices (see for example \cite{Reiner}). One can also see it as a particular case of the Cartan decomposition for algebraic groups (see Section \ref{preliminaries}).\\

In both cases,  given $g$ in $SL_d(k)$ its components in the $KAK$ decomposition are not uniquely defined   (only the component in $A$ is ). Nevertheless,  we can always fix once and for all a privileged way to construct $KAK$ in $SL_d(k)$.
  Therefore, for $g\in SL_d(k)$, we denote by $g=k(g)a(g) u(g)$ ``its'' KAK decomposition with $a(g)=diag\left(a_1(g),...,a_d(g) \right)$.\\ \smallskip
  Till the end of the paper, we write $v_g=k(g)[e_1]$ and  $H_{g}=\big[Span \langle u(g)^{-1}e_2,...,u(g)^{-1} e_d \rangle \big]$.
  The following lemma taken from \cite{densefree} shows that a large ratio between $a_1(g)$ and $a_2(g)$ implies contraction. Then $v_g$ can be taken as an attracting point and $H_g$ as a repelling hyperplane.
 \begin{lemme}\cite{densefree}
Let $\epsilon >0$. If $|\frac{a_2(g)}{a_1(g)}|\leq {\epsilon^2}$, then $[g]$ is $\epsilon$-contracting.
Moreover, one can take $v_g $ to be the attracting point and $H_g$  to be the repelling hyperplane.  \label{cruciallemme}\end{lemme}
 \begin{proof} $v_g=[k(g)e_1]$ and $H_g=\big[Span\langle u(g)^{-1}e_2,...,u(g)^{-1}e_d\rangle \big]$. Let $x\in V$ such that $d(x,H_g)>\epsilon$. We want to prove that $d(g[x],v_g)<\epsilon$.
Notice that $H_g=Ker\left(u(g)^{-1}.e_1^*(.)\right)$. Hence $\frac{|u(g)^{-1}.e_1^*(x)|}{||x||}>\epsilon$. But,$$d(g[x],v_g)=\frac{|| g x \wedge k(g)e_1||}{||gx||}= \frac{||a(g) u(g) x \wedge e_1||}{||a(g) u(g) x||}$$
Since $|a_1(g)| \geq....\geq |a_d(g)|$,\; $||a(g) u(g) x \wedge e_1|| \leq |a_2(g)| ||x||$. Moreover,
 $||a(g) u(g)x||\geq |a_1(g)|\;|u(g)^{-1}.e_1^* (x)|$.
 Hence, $$d(g[x],v_g)\leq \frac{|a_2(g)|}{|a_1(g)|}\;\frac{1}{\delta(x,H_g)}< \epsilon$$\end{proof}

\section{Random matrix products in local fields}
\label{subsproba}
$\bullet$ In this section, $d$ is an integer $\geq 2$ and $k$ a local field. We set $V=k^d$.\\

$\bullet$ When $\mu$ is a probability on a group $G$, we consider both random walks  $M_n=X_1...X_n$ and $S_n=X_1...X_n$ as defined in Section \ref{strategy}. $\Gamma_\mu$ is the smallest closed semigroup containing the support of $\mu$.

\subsection{Introduction}
Our aim in this section is to establish the basics of the theory of random matrix products
over local fields. The section is structured as follows.

In Section \ref{subsconv}, we generalize the first principles and tools of random matrix theory to all
local fields.
In particular we establish the exponential convergence in direction (Theorem \ref{direction}) and
the exponential decay of
the probability of hitting a hyperplane (Theorem \ref{hausdweak}). A key ingredient in the proofs is
our cocycle lemma,
Lemma \ref{cocycle}, which is a rather general statement giving control on the size of a cocycle
in an abstract context.
Another important tool will be Proposition \ref{largedeviations}, which compares the size of the norm of the
random walk with the size
of the random walk applied to any fixed vector. It can be viewed as a weak form of Le
Page's large deviations
theorem (\cite[Theorem 7]{Page}) in the context of local fields. Making use of these
two ingredients, we then compare the $A$-component of
the random walk in the Iwasawa decomposition with the $A$-component in the Cartan
decomposition (Proposition
\ref{comparaison2}).

In Section \ref{preliminaries}, we review some basic facts about algebraic groups, absolutely irreducible
linear representations of semi-simple
algebraic groups over local fields and their classification through the highest weight
theory.

In Section \ref{subsestimate} and Section \ref{subsequi}, we establish limit theorems for the components of the Cartan
decomposition of the random walk. The
main results are Theorem \ref{iwabis} (exponential contraction of the A-component), Theorem \ref{expokak}
(exponential convergence
of the $K$-components) and Theorem \ref{independence} (asymptotic independence of the K-components). Our
method consists in
investigating the Iwasawa decomposition first by  proving the exponential contraction
of the $A$-component of the Iwasawa decomposition (Theorem \ref{iwa}).  In fact, in
order to study the Cartan
decomposition in the ambient $SL_d(k)$, we will first look at the behavior of the Cartan
decomposition
of the random walk inside the semi-simple algebraic group which is the Zariski closure of
the group generated by the
random walk, and then compare the two decompositions (Corollary \ref{hardini}). The case when the
Zariski closure is connected
is easier and is dealt with in Section \ref{subsestimate}, while the general case is handled in Section \ref{subsequi}.

\subsection{Convergence in direction}
\label{subsconv}
\subsubsection{Generalization of well-known results in an non archimedean setting}
\label{gene}
This section does not require any prior knowledge on algebraic groups.\\
 Let  $B=(e_1,...,e_d)$ be the canonical basis of $V=k^d$. By canonical norm, we mean either
the standard Euclidean (or Hermitian) norm when $k$ is archimedean or the Max norm,
$||x||=Max\{|x_i|; i=1,...,d\}$ for every $x\in V$, when $k$ is non archimedean.\\
Recall that by Section \ref{generating}, there exist a compact subgroup $K$ acting by isometries on $V$, a subgroup $A^+$ consisting of diagonal matrices such that:
 $SL_d(k)=KA^+K$ (Cartan decomposition). For $g\in SL_d(k)$, we denote by $g=k(a)a(g)u(g)$ a privileged decomposition of $g$ in this product.\\

We denote by $V^*$ the dual of $V$ and  $(e_1^*,...,e_d^*)$ the canonical basis of  $V^*$ dual to $(e_1,...,e_d)$. We consider the canonical norm induced on $V^*$.  Recall that $SL_d(k)$ acts on $V^*$ by $g.f(x)=f(g^{-1}x)$ for every $g\in SL_d(k)$, $f\in V^*$, $x\in V$.
The projective space of $V$ is denoted by $P(V)$ and the projection of a non zero vector $x\in V$ by $[x]$.
The norm on $V$ (resp. $V^*$) induces a distance on $P(V)$ sometimes called the Fubini-Study distance:
$$\delta([x],[y])=\frac{||x \wedge y||}{||x||||y||}\;\;\;;[x],[y]\in P(V)$$
A similar formula holds for $V^*$. If $H$ is a hyperplane of $V$, $f\in V^*$ such that $H=Ker(f)$, then
$$\delta([x],H)=\frac{||f(x)||}{||f||||x||}\;\;\;\;;x\in V\setminus\{0\}$$

Consider a probability measure $\mu$
on $SL_d(k)$. No assumptions will be made on the Zariski closure of $\Gamma_\mu$. Recall that $M_n=X_1...X_n$ and $S_n=X_n...X_1$.  The KAK decomposition of $S_n$ will be simply denoted by: $S_n=K_nA_nU_n$.

\begin{defi} If $G$ is a group acting on a topological space $X$, $\mu$ (resp. $\nu$) a probability measure on $G$ (resp. $X$), $\nu$ is said to be $\mu$-invariant if $\mu \star \nu=\nu$, which means that for every borel function on $X$,\; $\iint{f(g.x) d\mu(g)d\nu(x)}=\int{f(x)d\nu(x)}$.
 \end{defi}

\begin{defi} [Lyapunov exponents] Suppose that $\int{\log ||g|| d\mu (g)}< \infty $ (i.e. existence of a moment of order one ).  The Lyapunov exponents relative to $\mu$ are defined recursively by:
$$\lambda_1+...+\lambda_i = \; \lim {\frac{1}{n} \E (\log ||\bigwedge ^i S_n ||)}=\:\lim {\frac{1}{n} \log ||\bigwedge ^i S_n ||}$$
\label{lyaponuv} \end{defi}

The limit on the left hand side is an  easy application of the subadditive lemma. The one on the right hand side is an almost sure limit and its existence is guaranteed by the subadditive ergodic theorem of Kingman \cite{Kingman}.
\begin{defi}[Index of a semigroup]
For any  semigroup $\Gamma$ of $GL(V)$, we define its index  as the least integer $p$ such that there exist sequences $\{M_n;n \geq 0\}$ in $\Gamma$,  $\{r_n;n \geq 0\}$ in $k$ such that $||r_nM_n||=1$, for which $r_nM_n$ converges to a rank $p$ matrix. We say that $\Gamma$ is contracting when the index is one. (Note that in the archimedean case, one can just look at the quantity $\frac{ M_n}{||M_n||}$).
\end{defi}

We begin by a fundamental lemma in this theory due to Furstenberg.
\begin{lemme} \cite{Furst}
Let $G$ be a topological semigroup acting on a $2^{nd}$ countable locally compact space $X$. Consider a sequence $\{X_n, n\geq 1\}$ of independent random elements of $G$ with a common distribution $\mu$ defined on $(\Omega, A, \p)$. We denote $\lambda=\sum_{n=0}^\infty {2^{-n-1} \mu^n}$. If $\nu$ is a $\mu$-invariant probability measure on $X$ then there exists a random probability measure $\nu_\omega$ on $X$ such that for $\p \otimes \lambda$-almost every $(\omega,g)$, the sequences of probability measures $X_1(\omega)...X_n(\omega)g\;\nu$ converge weakly to $\nu_\omega$ as $n$ goes to infinity.
\label{furst}  \end{lemme}

Using  Lemma \ref{furst}, Guivarc'h and Raugi proved in their fundamental work in \cite{Guivarch} the following crucial two theorems in the archimedean setting.  For a nice exposition of these results (over $\R$ or $\C$)  one can see chapter III of the book of Philippe Bougerol and Jean Lacroix \cite{bougerol}. We claim that these theorems hold in an arbitrary local field.  For the reader's convenience, we will check this for  the first theorem and assume it for the second one since the proof is just cutting and pasting their original proof (for example one can see pages 64-65 of \cite{bougerol}).
\begin{theo} \label{normalise} Suppose that  $\Gamma_\mu$ is strongly irreducible. Then, for $p$=index($\Gamma_\mu$), there exists  a random subspace  $V(\omega)$ of $V$ of dimension $p$  such that: a.s. for every $(r_n)_{n\in \N^*} \in k^\N$ s.t. $||r_n M_n ||=1$,  every limit point  of $r_n M_n$ is a rank $p$ matrix with image $V(\omega)$. Moreover  for every  $f\in V^*$, $$\p \left( f|_{ V(\omega)} \equiv 0 \right)=0$$
When $\Gamma_\mu$ is contracting, $p=1$ and there exists a unique $\mu$-invariant probability measure on the projective space $P(k^d)$ and a.s., $M_n(\omega) \nu$ converges weakly to $\delta_{Z(\omega)}$ where $Z$ is a random variable on $P(k^d)$ with law $\nu$. \end{theo}

\begin{theo} Suppose that $\int{\log ||g|| d\mu (g)}< \infty $. Under the same assumptions as in the previous theorem, $\lambda_1>\lambda_2$. \label{separation}\end{theo}

\textbf{Proof of Theorem \ref{normalise}:}
A general lemma of Furstenberg (see for example \cite{bougerol}, Proposition 2.3 page 49) says that every $\mu$-invariant probability measure on $P(V)$ is proper, i.e. does not charge any projective hyperplane. Now, fix a $\mu$-invariant probability measure on $P(V)$ and an event $\omega\in \Omega$. Choose $\{r_n; n\geq 1\}$ in $k$ such that $||r_n M_n(\omega)||=1$ and a limit point $A(\omega)$ along a subsequence $(n_k)_{k\in \N}$  of $\{r_nM_n; n \geq 1\}$. Hence for every $x\in V$ such that $x\not\in Ker(A(\omega))$, $M_{n_k}(\omega).[x]$ converges to $A(\omega).[x]$. Since $\nu$ is proper, we deduce that $M_{n_k}(\omega) g \nu$ converges weakly towards $A(\omega)g \nu$ for every $g\in SL_d(k)$. On the other hand,  by Lemma \ref{furst}, there exists a random probability measure $\nu_\omega$ on $P(V)$   (whose expectation is $\nu$) such that $M_n(\omega) g \nu$ converges weakly towards $\nu_\omega$ for $\lambda$-almost
every $g\in SL_d(k)$, where $\lambda$ is a probability measure supported on $\Gamma_\mu \cup \{I_d\}$.
By uniqueness of convergence in weak topology, $A(\omega) g \nu=\nu_\omega$ for $\lambda$-almost every $g\in SL_d(k)$. But $\{g\in SL_d(k); A(\omega)g \nu = \nu_\omega\}$ is closed and the support of $\lambda$ is $\Gamma_\mu \cup \{I_d\}$, hence
\begin{equation}A(\omega) g \nu=\nu_\omega\;\;\;\;\forall g\in \Gamma_\mu \cup \{I_d\}\label{klop}\end{equation}
Let $V(\omega)$ be the linear span of $\{x\in V; [x]\in Supp(\nu_\omega)\}$. (\ref{klop}) applied to $g=I_d$ shows that the image of $A(\omega)$ is exactly $V(\omega)$. Therefore, the image of $A(\omega)$ is indeed independent from the subsequence taken. It is left to show that its dimension is exactly the index $p$ of $\Gamma_\mu$. By definition of the index, the rank of $A(\omega)$ is at least $p$. The index of $\Gamma_\mu$ being $p$, there exists $\{h_n; n\geq 1\}$ in $\Gamma_\mu$, $\{s_n;n\geq 1\}$ in $k$ such that $s_nh_n$ converges to an endomorphism $h$ of rank $p$. (\ref{klop}) shows that:
$$A(\omega) gh_n \nu=\nu_\omega \;\;\;\;\forall g\in \Gamma_\mu; \; n\geq 1$$
We claim that one can find $g\in \Gamma_\mu$ such that: $$A(\omega)gh\nu=\nu_\omega$$ This would end the proof because the dimension of $V(\omega)$ would be less or equal to the range of $h$, which is $p$. It suffices to show that there exists $g\in \Gamma_\mu$ such that $\nu\{x\in V; A(\omega)ghx=0\}=0$, because in this case for $\nu$-almost every $[x]\in P(V)$, $A(\omega)gh_n[x]$ would converge to $A(\omega)gh[x]$ so that $\nu_\omega=A(\omega)gh_n \nu$ would converge to $A(\omega)gh\nu$. If on the contrary, for every $g\in \Gamma_\mu$, $\nu\{x\in V; A(\omega)ghx=0\}>0$, then by the aforementioned property of $\nu$, $$A(\omega)ghx = 0\;\;\;\forall x\in V$$
Hence $\{gx; g\in \Gamma_\mu; x\in Im(h)\}$ would be contained in the kernel of $A(\omega)$. Since it is $\Gamma_\mu$-invariant, this contradicts the irreducibility assumption on $\Gamma_\mu$. We have then proved that $V(\omega)$ is a $p$-dimensional subspace of $V$ and is the image of every limit point of $r_nM_n$, where $||r_nM_n||=1$.  By Lemma \ref{furst}, $\nu=\int{\nu_\omega\;d\p(\omega)}$. Therefore,

   \begin{eqnarray}\p (f|_{V(\omega)}\equiv 0)&= &\p \left( f(y)=0\; \forall y\in Supp(\nu_\omega) \right)\nonumber\\
&\leq& \E \left( \int {\mathds{1}_{f(y)=0} \;d_{\nu_\omega([y])}  }\right)\nonumber\\
&= &\nu \left(Ker(f) \right)\nonumber\end{eqnarray}
Since $\nu$ is proper, this is equal to zero. \\
Finally, if $\Gamma_\mu$ is contracting, then $p=1$ by definition and $[V(\omega)]$ is reduced to a point $Z(\omega)\in P(V)$. Since, by Lemma \ref{furst}, $\nu=\int{\delta_{Z(\omega)} \;d\p(\omega)}$, we deduce that the distribution of $Z$ is $\nu$ and hence $\nu$ is unique. \begin{flushright} $\Box$ \end{flushright}

\begin{corollaire}[Convergence in KAK]\label{essentiell} Suppose that $\Gamma_\mu$ acts strongly irreducibly on $V$. Then   the subspace $\left(k(M_n)e_1,...,k(M_n)e_p \right)$ converges a.s. to a random subspace $V(\omega)$ of dimension $p=\textrm{index}(\Gamma_\mu)$. Similarly, the same holds for the subspace $(U_n^{-1}.{e_1}^*,...,{U_n^{-1}}.{e_p}^*)$. Moreover, a.s. $\lim_{n\rightarrow \infty}{\frac{a_{p+1}(M_n)}{a_1(M_n)}} = 0$ and $Inf_{n} {\frac{a_p(M_n)}{a_1(M_n)}} >0$. The latter two assertions hold for $S_n$. \end{corollaire}
\begin{remarque} It is clear that we can replace $U_n^{-1}. e_1^*$,...,$U_n^{-1}.e_p^*$ with $U_n^{t}e_1$,...,$U_n^te_p$ where $U_n^t$ is the transpose of the matrix $U_n$. However, we prefer to work with the action on the dual vector space because it will give us more freedom later on. \label{dual}\end{remarque}
\begin{proof}
Let $a_1(M_n),...,a_d(M_n)$ be the diagonal  components of $a(M_n)$. Since $K$ acts by isometries on $V$, $|a_1(M_n)|=||M_n||$. Hence, for $p$=index ($\Gamma_\mu$),  Theorem \ref{normalise} gives  a $p$-dimensional (random) subspace $V(\omega)$ which is the range of every limit point of $\frac{M_n}{a_1(M_n)}$. Fix a realization $\omega$, we have:
 $$\frac{M_n (\omega)}{a_1(M_n(\omega))}= k(M_n (\omega)) \; diag\left(1,..., \frac{a_d(M_n(\omega))}{a_1(M_n(\omega))}\right) \;u(M_n(\omega))$$
Each component in this equation lies in a compact set. If $A(\omega)$, $K_\infty (\omega)$, $U_\infty (\omega)$,  $\alpha_2(\omega),...,\alpha_d (\omega)$ are limit points of $\frac{M_n}{a_1(M_n)}$, $k(M_n(\omega))$, $u(M_n(\omega))$, $\frac{a_2(n)}{a_1(n)}$,...,$\frac{a_d(n)}{a_1(n)}$, then $$A(\omega)=K_\infty (\omega) diag \left( 1,...,\alpha_d(\omega)\right) U_\infty(\omega)$$
Since $A(\omega)$ is almost surely of range $p$, almost surely, $\alpha_{p+1}(\omega)=...=\alpha_d(\omega)=0$ and $\alpha_2(\omega),...,\alpha_p(\omega)$ are non zero elements of $[0,1]$ when $k$ is archimedean and of $\Omega_k$ when $k$ is non archimedean; proving the last assertion of the corollary.\\  Since the image of $A(\omega)$ is $V(\omega)$,
$$V(\omega) \subset  Span \langle  K_\infty (\omega)e_1,..., K_\infty (\omega)e_p\rangle $$
By equality of dimension, we deduce that the two subspaces above are almost surely equal. As this holds for any convergent subsequence, we have the convergence a.s. of the subspace $\left(k(M_n) e_1,...,k(M_n) e_p \right)$ towards $V(\omega)$. \\
Now notice that $\Gamma_\mu$ acts strongly irreducibly on $V$ if and only if  $\Gamma_{\mu^{-1}}$ acts strongly irreducibly on $V^*$. Moreover, $\Gamma_\mu$ has the same index as $\Gamma_{\mu^{-1}}$ viewed as a subgroup of $SL(V^*)$ (it is just formed by the transposed matrices of $\Gamma_\mu$). Hence the same proof as above holds by looking at  $S_n^{-1}=X_1^{-1}...X_n^{-1}$ acting on $V^*$-  instead of $M_n=X_1...X_n$  acting on $V$.
\end{proof}
\begin{prop}\label{laprop}  If $\Gamma_\mu$ acts strongly irreducibly on $V$, then for any  sequence $\{x_n;n\geq 0\}$ in $V$ converging to a non zero vector:
\begin{equation}\textrm{a.s}\;\;\;\;\;\;\;\;inf_{n\in \N^*}{\frac{||S_n.x_n||}{||S_n||}} >0\label{inter}\end{equation}
\end{prop}

\begin{proof}
Let $S_n=K_nA_nU_n$ be a KAK decomposition and $(x_n)_{n\in \N}$  a sequence in $V$ converging to some $x\neq 0$.\\

\textbf{When $k$ is archimedean}:
To keep the exposition as simple as possible,  we will work here with the transpose matrices instead of working on the dual vector space: for $g\in SL_d(k)$, $g^*$ will denote its transpose (resp. conjugate transpose) matrix when $k=\R$ (resp. $k=\C$).
\begin{eqnarray}{{\frac{||S_n.x_n||^2}{||S_n||^2}}}= {\frac{||A_n U_n x_n||^2}{||A_n||^2}}= \frac{\sum_{i=1}^d{a_i(n)^2|<U_nx_n,e_i>|^2}}{a_1(n)^2}\geq  \; \left(\frac{a_p(n)}{a_1(n)}\right)^2\;{\sum_{i=1}^p {|<x_n,U_n^*e_i>|^2}} \nonumber\end{eqnarray}
 By Corollary \ref{essentiell}, a.s. $inf_{n\in \N^*}\;\frac{a_p(n)}{a_1(n)} >0$. \\
 We claim  that a.s.
 \begin{equation}Inf_{n\in \N^*}\;\sum_{i=1}^p {|<x_n,U_n^*e_i>|^2} \;>\;0\label{sakenalbii}\end{equation}

Indeed, by Corollary \ref{essentiell}, the subspace $(U_n^*e_1,...,U_n^*e_d)$ converges a.s. to a subspace $V(\omega)$. Let $\Pi_{V(\omega)}$ be the orthogonal projection on $V(\omega)$.
Hence $\sum_{i=1}^p {|<U_n^*e_i,x_n>|^2} \overset{\textrm{a.s.}}{{\underset{n \rightarrow \infty}{\longrightarrow}}} ||\Pi_{V(\omega)}(x)||^2 $. By Theorem \ref{normalise}:
$\p \left( \Pi_{V(\omega)} (x)=0 \right)=0$. The claim is proved.

\textbf{When $k$ is non archimedean},

$${\frac{||S_n x_n ||}{||S_n||}}\;= \frac{1}{|a_1(n)|}\;Max \{ |a_i(n)| |U_n^{-1}.e_i^*(x_n)|\;; i=1,...,d\} \;
 \geq \frac{|a_p(n)|}{|a_1(n)|} \; Max \{ |U_n^{-1}.e_i^* (x_n)|;\; i=1,...,p \} $$
Again, by Corollary \ref{essentiell}, $inf_{n\in \N^*}{\frac{|a_p(n)|}{|a_1(n)|}} > 0$ and it suffices to show that, a.s,
  \begin{equation}Inf_{n\in \N^*}\; Max \{ |U_n^{-1}.e_i^* (x_n)|;\; i=1,...,p \}> 0 \label{kkk}\end{equation}
 Indeed, let $V(\omega)$ be the limiting subspace  of $(U_n^{-1}.e_1^*,...,U_n^{-1}.e_p^*)$ and $U_\infty$  a limit point of $U_n$.  $Max \{ |U_n^{-1}.e_i^*  (x_n)|;\; i=1,...,p \}$ converges then a.s.,  via a subsequence,  to $Max \{|(U_\infty )^{-1}. e_i^*) (x)|;\;i=1,...,p\}$. The following claim shows that this is in fact independent from the subsequence  and equals\; $Sup\{\frac{|f(x)|}{||f||};\; f\in V(\omega)\} $,  which is a.s. positive because by Theorem \ref{normalise}, $\p \left(f(x)=0\; \forall f\in V(\omega) \right)=0$. \\
\;\;\textbf{Claim }: Let $V$ be a vector space of dimension $d\geq 2$ with basis $(e_1,...,e_d)$, $E$ a subspace of the dual $V^*$ of dimension $p<d$, $B=(f_1,...,f_p)$ a basis of the dual $E$. We suppose that $B$ is in the orbit of $(e_1^*,...,e_p^*)$ under the natural action of $K=SL_d(\Omega_k)$ on $(V^*)^p$. In other words, assume that there exists $g\in K$ such that $f_i=ge_i^* $ for every $i=1,...,p$. Then for every non zero vector $x\in V$ $$max\{ |f_i(x)|; i=1,...,p\} = Sup\{\frac{|f(x)|}{||f||};\; f\in E^*\}$$

\textbf{Proof of the claim}: let $f\in E^*$; $f=\sum_{i=1}^p {\lambda_i f_i}$, $\lambda_i \in k$. Since $|.|$ is ultrametric, \\ $|f(x)|\leq Max\{|\lambda_i|,i=1,...,p\} Max\{|f_i(x)|; i=1,...,p\}$. But, $f_i=ge_i^*$ with $g\in K$ which implies that $g^{-1}f = \sum_{i=1}^p {\lambda_i e_i^*}$ so that $||f||= ||g^{-1}f||=Max\{|\lambda_i|; i=1,...,p\}$. Hence $|f(x)|\leq ||f|| \; Max\{|f_i(x)|; i=1,...,p\}$.
 \end{proof}

 \begin{corollaire}Suppose that $\int{\log (||g||) d\mu(g)}< \infty$.  For any sequence $\{x_n;n\geq 0\}$ converging to a non zero vector $x$ of $V$; $$\frac{1}{n} \log ||S_n x_n|| \underset{n \rightarrow \infty}{\overset{\textrm{a.s.}}{\longrightarrow}} \lambda_1\;\;\;\;\;;\;\;\;\;\;
 Sup_{x \in V\setminus \{0\}}\;\frac{1}{n} \E (\log \frac{||S_n x||}{||x||})\underset{n \rightarrow \infty}{\longrightarrow} \lambda_1$$\label{coro}\end{corollaire}
 \begin{proof} The convergence on the left hand side is an immediate application of last proposition and the definition of the Lyapunov exponent. For the right hand side, by compactness of $P(V)$, it suffices to show that for any sequence $\{x_n;n\geq 0\}$ in the unit sphere converging to a non zero vector $x$ of $V$:
$ \frac{1}{n} \E (\log \;||S_n x_n||)\underset{n \rightarrow \infty}{\longrightarrow} \lambda_1$. By independence and equidistribution of the increments and by the inequality $||g||\geq 1$ true for every $g\in SL_d(k)$ we get: $\frac{1}{n} |\;\log ||S_n x_n||\;|\leq \frac{1}{n}\sum_{i=1}^n \log ||X_i||$. By the moment assumption on $\mu$, we can apply the strong law of large numbers which shows that the right hand side of the latter quantity converges in $L^1$  and is  consequently uniformly integrable. A fortiori, $\{\frac{1}{n} \log ||S_n x_n||; n\geq 0\}$ is uniformly integrable. Since it converges in probability (by the law of large numbers), we deduce that it converges in $L^1$.
 \end{proof}

\subsubsection{A cocycle lemma  - Application 1: ``weak'' large deviations}
\begin{defi} Let $G$ be a semigroup acting on a space $X$. A map $G\times X \overset{s}{\longrightarrow} \R$ is said to be an additive cocycle if  \; $s (g_1g_2,x)=s(g_1,g_2.x)+s(g_2,x)$\; for any $g_1,g_2\in G$, $x\in B$. \end{defi}
\begin{lemme}[Cocycle lemma] \label{cocycle}
Let $G$ be a semigroup acting on a space $X$, $s$ a cocycle on $G\times X$, $\mu$ a probability measure on $G$ satisfying for $r(g)=sup_{x\in X}{|s(g,x)|}$: there exists  $\tau>0$ such that  \begin{equation}\label{condition}\E\left(exp(\tau r(X_1))\right)< \infty\end{equation}
$\bullet$ If $$\lim_{n\rightarrow\infty} \;\frac{1}{n}Sup_{x\in X}\;{\E(s(S_n,x))}< 0,$$
then there exist $\lambda>0$, $\epsilon_0>0$, $n_0\in \N^*$ such that for every  $0<\epsilon< \epsilon_0$ and $n>n_0$:
$$Sup_{x\in X}\;\E\big[\;exp[\;\epsilon\left(s(S_n,x) \right)\;]\; \big] \leq (1-\epsilon \lambda)^n$$
$\bullet$ If $$\lim_{n\rightarrow\infty} \;\frac{1}{n}Sup_{x\in X}\;{\E(s(S_n,x))}= 0,$$
then for all $\gamma>0$, there exist $\epsilon(\gamma)>0$, $n(\gamma)\in \N^*$ such that for every $0<\epsilon<\epsilon(\gamma)$ and  $n>n(\gamma)$,
$$Sup_{x\in X}\;\E\big[\;exp[\;\epsilon\left(s(S_n,x) \right)\;]\; \big] \leq (1+\epsilon \gamma)^n. $$
\end{lemme}

\begin{remarque} The limit $\lim_{n\rightarrow\infty} \;\frac{1}{n}Sup_{x\in X}\;{\E(s(S_n,x))}$ always exists by sub-additivity \end{remarque}

\begin{proof} Let  $\epsilon>0$ and $Q_n=Sup_{x\in X}\;\E\Big[exp[\epsilon\left(s(S_n,x) \right)]\Big]$.
$Q_n$ being sub-multiplicative, for every $p$,
$$\limsup_{n \rightarrow \infty} \frac{1}{n} \log\; Q_n \leq \frac{1}{p} \log\;Q_p$$
Using the inequality $$exp(x)\leq 1+x+\frac{x^2}{2}exp(|x|)\;\;;x \in \R$$ we get for $\tau'=\frac{\tau}{3}$, $0\leq\epsilon\leq \tau'$,
 $$\limsup_{n \rightarrow \infty} \frac{1}{n} \log\; Q_n \leq \frac{1}{p} \log \Big( 1+ \epsilon \underset{a_p}{\underbrace{Sup_{x\in X}{\E(s(S_p,x))}}}\; + \frac{\epsilon^2}{2\tau'} \E \big(exp\left(\tau r(S_p) \right)\big)\Big)$$
Let $C=\E\big(exp\left(\tau(r(X_1))\right) \big)< \infty$. The cocycle property implies that $r(g_1g_2)\leq r(g_1)+ r(g_2)$ for every $g_1,g_2\in G$, whence $\E \big( exp\left(\tau(r(S_p))\right) \big)\leq C^p$. Hence, for every integer $p$,
\begin{equation}\limsup_{n \rightarrow \infty} \frac{1}{n} \log\; Q_n \leq \frac{1}{p} \log\; \left( 1+ \epsilon a_p\; + \frac{\epsilon^2}{2\tau'} C^p \right)\label{estefan}\end{equation}

The following inequality being true for every $x\in [-1;\infty[$:
$$(1+x)^{\frac{1}{p}}\leq 1+\frac{x}{p}$$
(\ref{estefan}) becomes: for every integer $p$,
\begin{equation}\limsup_{n \rightarrow \infty} \frac{1}{n} \log\; Q_n \leq  \log\; ( 1+ \epsilon \frac{a_p}{p} + \frac{\epsilon^2}{2\tau'} \frac{C^p}{p} )\label{gemma}\end{equation}

$\bullet$ Suppose first that $\frac{a_p}{p}$ converges to $\lambda'<0$  as $p$ goes to infinity. The quantity $a_p$ being sub-additive,  $\frac{a_p}{p}$ converges to $inf_p\; \frac{a_p}{p}$, hence $inf_p\; \frac{a_p}{p}=\gamma'< 0$. Then,    for some $p_0$, $a_{p_0}<0$. Put $\lambda = - \frac{ a_{p_0}}{2p_0}\;>0$. Apply  (\ref{gemma}) with $p=p_0$ and choose $\epsilon>0$ small enough such that: $\frac{ a_{p_0}}{p_0}\epsilon  + \epsilon^2 \frac{C^{p_0}}{2\tau' p_0} \leq - \lambda \epsilon$ \;$\Longleftrightarrow$ \;$0<\epsilon \leq \frac{- \tau' a_{p_0}}{C^{p_0}}$.\\

$\bullet$ Suppose  that $\frac{a_p}{p}$ converges to zero  as $p$ goes to infinity. \\
Fix $\gamma>0$.
Since $\lim {\frac{a_p}{p}}= 0$,\; for $p\geq p(\gamma)$ large enough,\; $\frac{a_p}{p}\leq \frac{\gamma}{2}$.\; Fix such $p$. For $\epsilon\leq \epsilon(\gamma)$ small enough, $\epsilon^2 \frac{C^p}{2\tau' p} \leq \epsilon \frac{\gamma}{2}$. It suffices now to apply (\ref{gemma}).
\end{proof}

\paragraph{Application1: ``Weak large deviations''}
In the real and complex cases, Le Page \cite{Page} proved a large deviation inequality for the quantities $\frac{1}{n}\log ||S_n||$ and $\frac{1}{n}\log ||S_n x||$, for any non zero vector $x$ of $V$.  By Proposition \ref{coro} these quantities converge towards the first Lyapunov exponent $\lambda_1$. More precisely, for every $\epsilon>0$,  there exist $\rho=\rho(\epsilon)\in ]0,1[$ and $n_0=n_0(\epsilon)$ such that for $n\geq n_0$, \begin{equation}\label{strong}\p \left(\big|\frac{1}{n}\log ||S_n|| - \lambda_1 \big|\geq \epsilon \right)\leq \rho^n \;\;;\;\;\p \left(\big|\frac{1}{n}\log ||S_n x|| - \lambda_1 \big|\geq \epsilon \right)\leq \rho^n\end{equation}
In particular, for some new  $\rho=\rho(\epsilon)\in ]0,1[$,  \begin{equation} \p \left(\frac{||S_n||}{||S_n x||} \geq exp(n\epsilon) \right)\leq \rho^n\label{weak}\end{equation}
This bound will be important for us later. Verifying Le  Page proof when $k$ is ultrametric is straightforward although somewhat lengthy. Alternatively  we will directly show   (\ref{weak}) using our cocycle Lemma \ref{cocycle}. Moreover our bound will be uniform in $x$ ranging over the unit sphere in $V$.

\begin{prop}[Weak large deviations]\label{largedeviations}
Suppose that $\mu$ has an exponential local moment and that $\Gamma_\mu$ is strongly irreducible. Then
for every $\gamma>0$, there exist $\epsilon(\gamma)>0$  and $n(\gamma)\in \N^*$  such that for $0<\epsilon< \epsilon(\gamma)$ and $n>n(\gamma)$:
\begin{equation}Sup_{x\in V;\;||x||=1} \;{\E\big[ (\frac{||S_n ||}{||S_n x||})^\epsilon\big]}\leq (1+\epsilon \gamma)^n \label{lessaliha}\end{equation}
In particular, for every $\epsilon>0$,\begin{equation}\label{awaida}\limsup_{n \rightarrow \infty} {\Big[ Sup_{x\in V;\;||x||=1}\;\p \left(\frac{||S_n||}{||S_nx||}\geq exp(n\epsilon) \right) \Big]^{\frac{1}{n}}} < 0\end{equation}
\end{prop}
\begin{proof} Let $\gamma>0$. First we prove that for $\epsilon < \epsilon(\gamma)$ and $n> n (\gamma)$,
\begin{equation} Sup_{[x],[y]} \;{\E \big[ (\frac{||S_n x || ||y||}{||S_n y|| ||x|| })^\epsilon\big]}\leq (1+\epsilon \gamma)^n \label{fr1}\end{equation}
Indeed,  $s(g,([x],[y]))= \log \;\frac{||g x|| ||y||}{||g y|| ||x||}$ defines an additive cocyle on $\Gamma_\mu \times (P(V)\times P(V))$, for the natural action of $\Gamma_\mu$ on $P(V)\times P(V)$. It suffices now to verify the hypotheses of Lemma (\ref{cocycle}).
Since for every $g\in SL_d(k)$,  $||g^{-1}||\leq ||g||^{d-1} $, $\E\left( exp ( \tau \;r(X_1) ) \right)\leq \E (||X_1||^{\tau}||X_1^{-1}||^{\tau})\leq \E ( ||X_1||^{\tau d} )$. This is finite for $\tau$ small enough because $\mu$ has an exponential local moment.  The condition (\ref{condition}) of Lemma \ref{cocycle} is then fulfilled.
It suffices now to show that $$\lim_{n\rightarrow \infty} {Sup_{[x],[y]} {\E\left( {s(S_n,([x],[y])) } \right)}} = 0$$ ($\leq 0$ suffices in fact). Since $P(V) \times P(V)$ is compact, it suffices to show that for any convergent sequences $(x_n)$ and $(y_n)$ in the sphere of radius one:$$\lim_{n\rightarrow \infty} \frac{1}{n}\big[ \E (\log \;||S_n x_n||) - \E (\log \;||S_n y_n||) \big]  = 0$$
This is true since by (the proof of ) Corollary \ref{coro}:
\begin{equation}\label{tensa}\lim_{n\rightarrow \infty} {\frac{1}{n}\E (\log {||S_n x_n||}) } =
\lim_{n\rightarrow \infty}  {\frac{1}{n}\E (\log {||S_n y_n||}) } = \lambda_1\end{equation}

Notice that $||g|| \asymp max\{||g.e_i||;\;i=1,...,d\}$ for every $g\in GL(V)$. Hence, $Sup_{[x]}\; {\E \big[(\frac{||S_n || ||x|| }{||S_n x||  })^\epsilon\big]} \preceq \sum_{i=1}^d {Sup_{[x]}\; {\E \big[(\frac{||S_n e_i || ||x||}{||S_n x||  })^\epsilon\big]}}$.  Applying (\ref{fr1}) shows (\ref{lessaliha}).\\
Finally, we prove (\ref{awaida}): let $\epsilon>0$, $\gamma>0$ to be chosen in terms of $\epsilon$. By (\ref{lessaliha}) and the Markov inequality there exist $\epsilon'(\gamma)>0,n(\gamma)>0$ such that for $0< \epsilon'<\epsilon'(\gamma)$ and $n>n(\gamma)$:

$$\p \left(\frac{||S_n||}{||S_nx||}\geq exp(n\epsilon) \right) \leq exp(-n\epsilon \epsilon') \E \big[\left( \frac{||S_n||}{||S_n x||}\right)^{\epsilon'}\big]\leq  exp(-n\epsilon \epsilon')(1+\gamma\epsilon')^n$$
Since $exp(-n\epsilon\epsilon')=exp(\epsilon \epsilon')^{-n} \leq \frac{1}{(1+\epsilon \epsilon')^n}$, it suffices to choose  $\gamma=\frac{\epsilon}{2}$.
\end{proof}

\subsubsection{Application 2: exponential convergence in direction}

\begin{prop} \label{contracexpo}Suppose that $\mu$ has an exponential local moment and that $\Gamma_\mu$ is strongly irreducible and contracting. Then there exist $\lambda>0$, $\epsilon_0>0$, $n_0\in \N^*$ such that for $0<\epsilon< \epsilon_0$ and $n>n_0$:
$$\E \left(\frac{\delta^\epsilon (S_n[x],S_n[y])}{\delta^\epsilon ([x],[y])} \right) \leq (1-\lambda \epsilon)^n $$
\end{prop}
\begin{proof} Let $X=P(V) \times P(V) \setminus \textrm{diagonal}$ and $s$  the application on $\Gamma_\mu\times X$ defined by: $$s\left(g,([x],[y]) \right) =  \log\;\frac{\delta(g[x],g[y])}{\delta([x],[y])}\;\;;g\in \Gamma_\mu;([x],[y])\in X$$
It is easy to verify that $s$ is an additive cocycle on $\Gamma_\mu\times X$ for the natural action of $\Gamma_\mu$ on $X$. It suffices now to check the hypotheses of Lemma \ref{cocycle}. \\

\noindent By definition of the distance $\delta$, we have for every $g\in SL_d(k)$, $([x],[y])\in X$, $\log\;\frac{\delta(g[x],g[y])}{\delta([x],[y])} \leq 2d\;\log||g||$.
Since $\mu$ has an exponential local moment, (\ref{condition}) of Lemma \ref{cocycle} is valid.
      It is left to check that we are in the first case of the lemma, i.e. $\lim \frac{1}{n} Sup_{([x],[y])\in X}\;\E \left(s(S_n,(x,y))\right) < 0$.
\begin{eqnarray} \frac{1}{n} Sup_{([x],[y])\in X}\;\E \left(s(S_n,(x,y))\right)  &\leq&  \frac{1}{n} Sup_{([x],[y])\in X}
\E\left( \log\;\frac{||\bigwedge^2 S_n x \wedge y ||}{||x \wedge y ||} \right)+\; \frac{2}{n} Sup_{[x]\in P(V)}\;\E \left(\log \frac{||x||}{||S_n x||} \right)\nonumber\\
&\leq& \frac{1}{n}\E (\log ||\bigwedge^2 S_n||) +\; \frac{2}{n} Sup_{[x]\in P(V)}\;\E \left(\log \frac{||x||}{||S_n x||} \right) \label{majda} \end{eqnarray}
 By definition of the Lyapunov exponent, $$\frac{1}{n}\E (\log ||\bigwedge^2 S_n||) \underset{n \rightarrow \infty}{\longrightarrow} \lambda_1+\lambda_2$$
 By (the proof of ) Corollary \ref{coro}, $$\frac{1}{n} Sup_{[x]\in P(V)}\;\E \left(\log \frac{||x||}{||S_n x||} \right)\underset{n \rightarrow \infty}{\longrightarrow} -\lambda_1$$
 Hence, $$\lim\; \frac{1}{n} Sup_{([x],[y])\in X}\;\E \left(s(S_n,(x,y))\right) \underset{n \rightarrow \infty}{\longrightarrow} \lambda_2 - \lambda_1$$
 Under the contraction and strong irreducibility assumptions on $\Gamma_\mu$, this is negative by Theorem \ref{separation}.
 \end{proof}
 We deduce the following
\begin{theo}[Exponential convergence in direction] With the same notations and assumptions as in the previous proposition, there exists a random variable $Z_1$ (resp. $Z_2$) on $P(V)$ - with law $\nu$ (resp. $\nu^*$),  the unique $\mu$-invariant probability measure  on $P(V)$ (resp. $\mu^{-1}$-invariant on $P(V^*)$) such that for some $\lambda>0$ and every $\epsilon>0$:
\begin{equation}Sup_{[x]\in P(V)}\;{\E \left( \delta^\epsilon (M_n[x],Z_1)   \right)}\leq (1-\lambda\epsilon)^n\label{3abali}\end{equation}
\begin{equation}Sup_{[f]\in P(V^*)}\;{\E \left( \delta^\epsilon (S_n^{-1}. [f],Z_2)   \right)}\leq (1-\lambda\epsilon)^n \label{abyad}\end{equation}
In particular, for every $[x]\in P(V)$ (resp. $[f]\in P(V^*)$),  $M_n[x]$ (resp. $S_n^{-1}.[f]$) converges almost surely towards $Z_1$ (resp. $Z_2$).

\label{direction}\end{theo}

\begin{proof}It suffices to prove (\ref{3abali}). Indeed, (\ref{abyad}) is the consequence of the fact that  the action of $\Gamma_\mu$ on $V$ is strongly  irreducible and contracting if and only if the action of $\Gamma_{\mu^{-1}}$ on $V^*$ is. Moreover, if (\ref{3abali}) and (\ref{abyad}) hold then $M_n[x]$ (resp. $S_n^{-1}.[f]$) converges a.s. towards $Z_1$ (resp. $Z_2$) by an easy application of the Markov inequality.\\

 Let $Z$ be the random variable on $P(V)$ obtained in Theorem \ref{normalise}.
 Let $\lambda>0$, $\epsilon>0$ small enough and $n\geq n_0$ given by the previous proposition.
Fix $k>n$, $[y],[x]\in P(V)$. The triangle inequality gives:
\begin{equation}{\E \left( \delta^\epsilon (M_n[x],Z)   \right)} \leq \underset{(I)}{\underbrace{{\E \left( \delta^\epsilon (M_n[x],M_k[y])   \right)}}}
+ {\E \left( \delta^\epsilon (M_k[y],Z)   \right)}\label{blanche}\end{equation}

Since $M_k[y]=M_n X_{n+1}...X_k[y]$, we condition by the $\sigma$-algebra generated by $(X_{n+1},...,X_k)$ and obtain by independence of the increments
:
\begin{eqnarray} (I)&=&\int{d\mu^{k-n}(\gamma)\; \E \left(\delta^\epsilon (M_n[x],M_n[\gamma y]) \right)} \nonumber\\
 & \leq & Sup_{[a],[b]}\;\E (\delta^\epsilon (M_n[a],M_n[b])) \leq (1-\lambda\epsilon)^n \label{nn}\end{eqnarray}
 Inserting (\ref{nn}) in (\ref{blanche}) gives for every $[y]\in P(V)$, $k>n\geq n_0$:
 $$Sup_{[x]} \E(\delta^\epsilon(M_n[x],Z)) \leq (1-\lambda\epsilon)^n + \E (\delta^\epsilon (M_k[y],Z))$$
Let  $\nu$ be the unique $\mu$-invariant probability measure on $P(V)$ (see Theorem \ref{normalise}). Integrating with respect to $d\nu ([y])$ the two members of the previous inequality and applying Fubini theorem,  we get for every $k>n\geq n_0$:
 \begin{equation}Sup_{[x]} \E(\delta^\epsilon(M_n[x],Z)) \leq (1-\lambda\epsilon)^n + \E \left(\int{\delta^\epsilon ([y],Z) \; d(M_k\nu) ([y])} \right)\label{jjj}\end{equation}
Again by Theorem \ref{normalise}, a.s. $M_k \nu$ converges weakly towards the dirac measure $\delta_Z$  when $k$ goes to infinity. For $w$ fixed and every $0<\epsilon\leq 1$, $\delta^\epsilon \left(\;.\;, Z(\omega)\right)$ is a continuous function on $P(V)$.  Hence,  $\int{\delta^\epsilon ([y],Z) \; d(M_k\nu)([y])}$ converges a.s. to $\delta^\epsilon(Z,Z)=0$  when $k$ goes to infinity. By the dominated convergence theorem,
$\E \left(\int{\delta^\epsilon ([y],Z) \; d(M_k\nu) ([y])} \right) \underset{k\rightarrow \infty}{\longrightarrow}0$. We conclude by letting $k$ go to infinity in (\ref{jjj}). Since $\epsilon \mapsto \delta^\epsilon (.,.)$ is decreasing, the corollary is true for every $\epsilon>0$.
\end{proof}

\subsubsection{Weak version of the regularity of invariant measure}
An important result in the theory of random matrix products is the regularity of the invariant measure $\nu$, under contraction and strong irreducibility assumptions:
\begin{theo} \cite{Guivarch3} $k=\R$. Consider the same assumptions as in Proposition \ref{contracexpo}, then there exists $\alpha>0$ such that: $$Sup\{\;\int{\delta^{-\alpha} ([x],H) d\nu([x])};\;\;\textrm{$H$ hyperplanes of $V$} \} < \infty$$ In particular, if $Z$ is a random variable on $P(V)$ with law $\nu$, then for every $\epsilon>0$:
\begin{equation}Sup\{\;\p \left(\delta (Z,H) \leq  \epsilon \right);\;\;\textrm{$H$ hyperplane of $V$}\}\;\leq C\epsilon^\alpha \label{adsl}\end{equation}
\label{hausdorf}\end{theo}
\noindent (\ref{adsl}) gives in particular \underline{for $k=\R$}: for every $0<t<1$:
 \begin{equation} \limsup_{n \rightarrow \infty}{\big[Sup\{\p \left(\delta(Z,[H]) \leq t^n \right);\;\;\; \textrm{$H$ hyperplanes of $V$}\}\big]}^{\frac{1}{n}} < 1\nonumber\end{equation}

The latter assertion will be important for us.
Proving Theorem \ref{hausdorf} in an arbitrary local field can be done along the same lines as Guivarch's proof over the reals. We will refrain from including the details of this proof here, since we will not need the full force of \ref{hausdorf}. Instead we give a direct proof of the last assertion, using our ``weak large deviation'' - Proposition \ref{largedeviations}.\\\\

\begin{theo}  Consider the same assumptions as in Proposition \ref{contracexpo}. Let $Z$ be a random variable with law $\nu$, the unique  $\mu$-invariant probability measure. Then, for all $t\in]0,1[$, \begin{equation}\limsup_{n \rightarrow \infty}{\big[Sup\{\p \left(\;\delta(Z,[H]) \leq t^n\;\right);\;\; \textrm{$H$ hyperplanes of $V$}\}\big]}^{\frac{1}{n}} < 1   \nonumber\end{equation}
\label{hausdweak}\end{theo}
Before proving the theorem, we begin with an easy but crucial lemma.
\begin{lemme}There exists a constant $C(k)$ such that for every  $f\in V^*$, a.s. there exists $i=i(n,\omega)\in \{1,...,d\}$ such that:
$|f(M_ne_i)|\geq C(k) ||M_n^{-1}.f||$
\label{eftakartfdelt}\end{lemme}
\begin{proof} When $k$ in archimedean, a.s. $||M_n^{-1}.f||^2=\sum_{i=1}^d{|M_n^{-1}.f (e_i)|^2}$. Take $C(k)=\frac{1}{\sqrt{d}}$.
When $k$ in non archimedean, the norm on $V^*$ is ultrametric. Hence, a.s. $||M_n^{-1}.f||=Max\{|M_n^{-1}.f (e_i)|;\;i=1,...,d\}$. The lemma is then valid for $C(k)=1$. \end{proof}

\textbf{Proof of Theorem \ref{hausdweak}:}
Let $H$ be a hyperplane of $V$, $f\in V^*$ such that $H=Ker(f)$. One can suppose $||f||=1$.
Let $A_{i}$ be the event ``$\{||f(M_n e_i)||\geq C(k) ||M_n^{-1}.f||\}$''. By the previous lemma, $\p (\cup_{i=1}^d {A_i}) = 1$.  Hence,
\begin{equation}\label{rida}
\p \left(\delta(Z,[H]) \leq t^n\right) \leq \sum_{i=1}^d {\p \left(\delta(Z,[H]) \leq t^n;\;\mathds{1}_{A_i}\right)}\end{equation}

By Theorem \ref{direction},  there exists  $\rho_1\in ]0,1[$ such that for all large $n$:\\
$$Sup_{[x]\in P(V)}\;{\E \left(\delta(M_n[x],Z) \right)}\leq \rho_1^n$$ This implies by the Markov inequality that for every $\rho_2\in ]\rho_1,1[$ and for all large $n$:
\begin{equation}\label{sissi}{\p \left(\delta(M_n[x],Z) \geq \rho_2^n \right)\leq (\frac{\rho_1}{\rho_2})^n };\;\; \forall x\in V\setminus\{0\}\end{equation}
On each event $A_i$, we apply inequality (\ref{sissi}) for $x=e_i$. Inserting this in (\ref{rida}) and using the triangle inequality, we get:
\begin{equation}\p \left(\delta(Z,[H]) \leq t^n\right) \leq \sum_{i=1}^d {\p \left(\delta(M_n[e_i],[H]) \leq \rho_2^n+t^n;\;\mathds{1}_{A_i}\right)}+\; d(\frac{\rho_1}{\rho_2})^n \label{miaofen}\end{equation}
On the event $A_i$,
\begin{equation}\label{yosra}\delta(M_n[e_i],[H]) = \frac{|f(M_ne_i)|}{||M_ne_i||}\geq C(k)\frac{||M_n^{-1}.f ||}{||M_ne_i||}\end{equation}
Inserting (\ref{yosra}) in (\ref{miaofen}) gives:
$$\p \left(\delta(Z,[H]) \leq t^n\right)\leq \sum_{i=1}^d {\p \left(\frac{||M_n^{-1}.f||}{||M_ne_i||} \leq \frac{\rho_2^n+t^n}{C(k)}\right)}+\; d(\frac{\rho_1}{\rho_2})^n$$
The following assertion clearly ends the proof:
for any $a \in ]0,1[$,
\begin{equation}\limsup_{n \rightarrow \infty} \;\big[ \p (\frac{||M_n^{-1}.f||}{||M_n x||} \leq a^n ) \big]^{\frac{1}{n}} < 1  \label{sylvana} \end{equation}
uniformly in $f\in V^*$ of norm one and $x\in V$ of norm one. \;Indeed,  the action of $\Gamma_{\mu^{-1}}$ on $V^*$ is  strongly irreducible and contracting. Hence we can apply Proposition \ref{largedeviations} by replacing $S_n=X_n...X_1$ with $M_n^{-1}=X_n^{-1}...X_1^{-1}$, $V$ with $V^*$. If $\rho^*$ denotes the contragredient representation of $G$ on $V^*$, then for any $a \in ]0,1[$,
$$\limsup_{n \rightarrow \infty} \;\big[ \p (\frac{||M_n^{-1}.f||}{||\rho^*(M_n^{-1}) ||} \leq a^n ) \big]^{\frac{1}{n}} < 1$$
uniformly in $x$ and $f$.
Since $\rho^*(M_n^{-1})$ is just the transpose matrix of $M_n$,
$||M_nx||\leq ||M_n||= ||\rho^*(M_n^{-1})|| $.  Then (\ref{sylvana}) is valid uniformly in $x$ and $f$.

\begin{flushright}
$\Box$
\end{flushright}

\subsection{Preliminaries on algebraic groups}
\label{preliminaries}
Till the end of the paper, $k$ is a local field, $\mathbf{G}$ is a $k$-algebraic group, $G=\mathbf{G}(k)$
are the $k$-points of $\mathbf{G}$. \textbf{We will assume  $G$ to be  $k$-split and its connected component semi-simple}.  However $G$ itself is not assumed Zariski-connected unless explicitly mentioned. In general if $\mathbf{H}$ is a $k$-algebraic group, $H$ will denote its group of $k$-points. The word ``connected'' will refer to the Zariski topology.\\

In this section, $\mathbf{G}$ is connected. For references, one can see \cite{Tits1} for the description of irreducible representations,  \cite{tit1}, \cite{tit2} or \cite{mac} for the Cartan and the Iwasawa decomposition.
\paragraph{Decompositions in algebraic groups}
Let $\mathbf{A}$ be a maximal $k$-torus of $\mathbf{G}$, $\mathbf{X(A)}$ be the group of $k$-rational characters of $\mathbf{A}$, $\Delta$ be the system of roots of $G$ restricted to $\mathbf{A}
$, which consists of the common eigenvalues of $\mathbf{A}$ in the adjoint representation.
We fix an order on $\Delta$ and denote by $\Delta^+$ the system of positive roots, $\Pi$ the system of simple roots (roots than cannot be obtained as product of two positive roots) and define $A^+=\{a\in A\;;\;|\alpha(a)|\geq 1\;;\;\forall \alpha\in \Delta^+\}$.
There exists a maximal compact subgroup $K$ of $G$ such that $$G=KA^+K\;\;\;\;\;\textrm{Cartan or $KAK$ decomposition}$$
  We denote by $\mathfrak{g}$ be the Lie algebra of $G$ over $k$ and define, for every $\alpha\in \Delta$,
   $\mathfrak{g}_\alpha=\{x\in \mathfrak{g}\;;\;Ad(a).x=\alpha(a)x\;\forall a\in A\}$. Let
     $\mathbf{N}$ be the unique connected subgroup of $\mathbf{G}$ whose Lie algebra
      is $\oplus_{\alpha \in \Delta^+}{\mathfrak{g}_\alpha}$; it is a maximal unipotent connected subgroup.
      Then the following decomposition, called Iwasawa or KAN decomposition, holds: $$G=KAN\;\;\;\;\;\textrm{Iwasawa or KAN decomposition}$$

\paragraph{Rational Representations of algebraic groups}
In the previous paragraph, we used only the adjoint representation of $G$. More generally, if $(\rho,V)$ is a $k$-rational irreducible representation of $G$, $\chi\in \mathbf{X(A)}$ is called a weight of $\rho$ if it is a common eigenvalue of $A$ under $\rho$.
We denote by $V_\chi$ the weight space associated to $\chi$ which  is $V_\chi=\{x\in V; \rho(a)x=\chi(a)x\;\forall\;a\in A\}$. Then $V=\oplus_{\chi\in \mathbf{X(A)}}{V_\chi}$.
The representation $\rho$ is characterized by
a particular weight $\chi_\rho$ called highest weight which has the following properties: \\
$\bullet$ every weight $\chi$ of $\rho$ different from
$\chi_\rho$ is of the form: $\chi=\frac{\chi_\rho}{\prod_{\alpha\in \Pi}{\alpha^{n_\alpha}}}$, where $n_\alpha \in \N$ for every
simple root $\alpha$.\\
$\bullet$ Every $x\in V_{\chi_\rho}$  is fixed by the subgroup $N$.\\
Let $\Theta_\rho=\{\alpha\in \Pi;\;\chi_\rho/\alpha \;\textrm{is a weight of $\rho$}\}$.
 \begin{prop}\cite{Tits1}For every $\alpha\in \Pi$, let $w_\alpha$ be the fundamental weight associated to $\alpha$.
Then the $k$-rational irreducible representation $(\rho_\alpha,V_\alpha)$ of $G$ whose highest weight is $w_\alpha$ (called fundamental representation) has a highest weight space of dimension one  and satisfies   $\Theta_{\rho_\alpha}=\{\alpha\}$.  \label{tits}\end{prop}

Every $k$-rational irreducible representation $\rho$ of $G$ can be obtained as a sub-representation of tensor products of fundamental representations  and $\chi_\rho$ is of the form $\prod_{\alpha\in \Pi}{w_\alpha^{s_\alpha}}$, with $s_\alpha\in \N$.
We record below a basic fact about root systems (\cite[\S 1.9 et 1.10]{bourbaki}).
\begin{prop} Every root $\alpha\in \Delta$ is of the form: $\alpha=\prod_{\beta\in \Pi}{w_\beta^{n_\beta}}$,
with $n_\beta\in \Z$, for every $\beta\in \Pi$.
\end{prop}

\paragraph{Good norm}
Let $\rho $ be a $k$-rational irreducible representation of $G$.
We wish to find a special basis and norm of $V$ such
that
$\rho(G)=\rho(K)\rho(A^+) \rho(K)$ (resp. $\rho(G)=\rho(K)\rho(A)\rho(N)$\;) is the restriction of a Cartan (resp. Iwasawa) decomposition of  $SL(V)$, i.e. $K$ acts by isometries on $V$, $A$ acts by diagonal matrices with
$\rho(A^+)\subset \{diag(a_1,...,a_d); |a_1|\geq |a_i|\; \forall i\neq 1\}$,
 $\rho(N)$ fixes  the first vector of the basis. \\
 To do that we begin with standard definitions borrowed from Quint \cite{quint}.
Let $V$ be a $k$-vector space.
When  $k$ is $\R$ (resp. $\C$), we say that a norm on $V$ is  good  if and only
 if it is induced by a Euclidian scalar product (resp. Hermitian scalar product).
 Now if $V$ is endowed with a good norm, a direct sum  $V=V_1\oplus V_2$ is good if and only if it is
  orthogonal with respect to the scalar product.
  When $k$ is non archimedean, we say that a norm on $V$ is good if and only if it is
   ultrametric, i.e., $||v+w||\leq Max\{||v||;||w||\}$ $\forall v,w\in V$.
   A direct sum $V=V_1\oplus V_2$ is good if and only if for every $v=v_1+v_2$, with $v_1\in V$, $v_2\in V$,
   $||v||=Max\{||v_1||, ||v_2||\}$. \\\\

Now let $(\rho,V)$ be $k$-rational irreducible representation of $G$ and $V=\oplus_\chi V_\chi$ its decomposition into weight spaces.
 We write $G=KAK$ its Cartan decomposition.
\begin{theo}\label{emilia}[\cite[\S 2.6]{Mostow1} for $k$ archimedean, \cite[Theorem 6.1]{Quint1} for $k$ non archimedean]\\
When $k=\R$ (resp. $\C$), there exists a scalar product (resp. Hermitian scalar product) on $V$ such
$\rho(K)$ acts by isometries on $V$and $\rho(A)$ is symmetric (resp. Hermitian).
The direct sum $V=\oplus_\chi V_\chi$ is good and $a\in A$ induces on each $V_\chi$ a homothety     of ratio $\chi(a)$. \\
When $K$ is non archimedean, there exists a $K$-invariant ultrametric norm on $V$ such that the $V_\chi$'s are in  good direct sum. The action of $a\in A$ on $V_\chi$ is by homothety   of ratio $\chi(a)$
\end{theo}
Such a norm is said to be $(\rho,A,K)$-good.
\begin{corollaire} \label{youssif}Let $(\rho,V)$ be a $k$-rational representation of $G$, $\chi_\rho$ its highest weight.
Then there exists a good norm $||.||$ on $V$ such that
 $$||\rho(g)|| = |\chi_\rho \left(a(g)\right)|\;\;;g\in G$$ And for every $x_\rho\in V_{\chi_\rho}\setminus\{0\}$,
 $$\frac{||\rho(g) x_\rho||}{||x_\rho||}=|\chi_\rho \left(\widetilde{a (g)}\right)|\;\;;g\in G$$
 where $a(g)$ (resp. $\widetilde{a(g)}$) is the $A^+$ (resp. $A$) - component of
 $g$ in the Cartan (resp. Iwasawa) decomposition. \end{corollaire}

\paragraph{Fubiny-Study norm:}  Consider a good norm on $V$ and a good direct sum: $V=V_1\oplus V_2$. Then, there exists a good norm on $\bigwedge^2 V$ such that the direct sum $\bigwedge^2 V_1 \oplus (V_1 \bigwedge V_2) \oplus \bigwedge^2 V_2$ is good. This induces the Fubini-Study distance $\delta$ on the projective space $P(V)$:
$$\delta([x],[y])=\frac{||x \wedge y ||}{||x|| ||y||}\;\;;\;\;[x],[y]\in P(V)$$

\paragraph{An example: $SL_d(k)$ (\cite{pla}) }

Here we consider $\mathbf{G}=\mathbf{SL_d}$. A maximal $k$-torus is $A=\{diag(a_1,...,a_d); \;\prod_{i=1}^d a_i=1\}$  and $A^+=\{diag(a_1,...,a_d)\in A; \;|a_1|\geq...\geq |a_d|\}$.

To simplify notations, for $i=1,...,d$, we denote by $\lambda_i$  the following rational character of $A$:
$(\lambda_1,...,\lambda_d) \mapsto \lambda_i$.
Simple roots are $ \lambda_i/\lambda_{i+1}$, $i=1,...,d-1$.  Positive roots are
$\lambda_i/\lambda_{j}$, $1\leq i< j \leq d$.
The fundamental weight associated to $\alpha_{i}=\lambda_i/\lambda_{i+1}$  is $w_i=\lambda_1...\lambda_i$
 and the representation $\rho_{\alpha_i}$ of Proposition \ref{tits} is just $\bigwedge^i V$.  The expression of simple roots in terms of fundamental weights is:

$$\alpha_i = w_{i-1}^{-1}.{w_i}^2.w_{i+1}^{-1}\;\;;\;\;i=1,...,d$$

Let $K=SO_d(\R)$ (resp. $K=SU_d(\C)$) when $k=\R$ (resp. $k=\C$) and $K=SL_d(\Omega_k)$ when $k$ is
 non archimedean.  We denote by $N$  the subgroup of upper triangular matrices with $1$ on the diagonal.
  Then the Cartan decomposition is $G=KA^+K$ and the Iwasawa decomposition: $G=KAN$. As seen in Section \ref{uff}, we can  also take the following other choice for $A^+$: $A^+=\{diag(a_1,...,a_d);\; a_i\in ]0;+\infty[;\;a_1\geq...\geq a_d>0; \prod_{i=1}^d {a_i}=1\}$ when $k=\R$ or $\C$ and $A^+=\{diag(\pi^{n_1},...,\pi^{n_d});\;n_1\leq...\leq n_d;\; \sum_{i=1}^d {n_i}=0\}$ when $k$ is non archimedean. Let $B=(e_1,...,e_d)$ be
 the canonical basis on $V$ and $||.||$  the canonical
 norm on $V$ (see Section \ref{gene}),  then it is clear that $K$ acts by isometries on $V=k^d$.\;
Consequently, $B$ is in a good direct sum and $||.||$ is $(A,K)$-good.

\subsection{Estimates in the
 Cartan decomposition - the connected case}
\label{subsestimate}

\textbf{In this section G is assumed Zariski-connected.} Recall that $\mathbf{G}$ is also assumed semi-simple and $k$-split.

Let $\mu$ be a probability measure on $G=\mathbf{G}(k)$ and $\rho$ a $k$-rational irreducible
representation of $G$ into some $SL_d(k)$. We assume $\Gamma_\mu$ to be Zariski dense in $G$.\\
Our aim in this section is  to give estimates of the Cartan decomposition in $\rho(G)$ of the
random walks  $\rho(M_n)$, $\rho(S_n)$  using their Iwasawa decomposition. \\

Let $\chi_\rho$ be the highest weight for $V$,
and $r$ the number of non zero weights of $V$.
 We set $\chi_1=\chi_\rho$, $\chi_2,...,\chi_l$ ($l\in \{2,...,r\}$) the weights adjacent to $\chi_1$, i.e., such that $\chi_i=\chi_1$ or there is $\alpha\in \Theta_\alpha$ such that $\chi_i=\chi_1/\alpha$.
  We consider a $(\rho,A,K)$-good norm on $V$ (for the  basis of weights) given by Theorem \ref{emilia} of the preliminaries.

For $g\in G$, we denote by $g=k(g)a(g)u(g)$ (resp. $g=\widetilde{k(g)}\widetilde{a(g)}\widetilde{n(g)}$) a privileged Cartan (resp. Iwasawa) decomposition in $G=KA^+K=KAN$. When it comes to the random walk $S_n=X_n...X_1$, we simply write $S_n=K_nA_nU_n$ (resp. $S_n=\widetilde{K_n} \widetilde{A_n} N_n$) for the KAK (resp. KAN) decomposition of $S_n$ in $G$ and set $\rho(A_n)=diag(a_1(n),...,a_d(n))\;\;;\;\; \rho(\widetilde{A_n})=diag(\widetilde{a_1(n)},...,\widetilde{a_d(n)})$.\\

It is known that $G$ is isomorphic to a closed subgroup of $GL_r(k)$ for some $r\geq 2$ - \cite{Humphreys}. Let $i$ be such an isomorphism. (When $G$ is simple and of adjoint type, one can take the adjoint representation).

\begin{defi}[Exponential moment for algebraic groups] \label{moshader}
If $\mu$ is a probability measure on $G$, we say that $\mu$ has an exponential local moment if
$i(\mu)$ (image of $\mu$ under $i$) has an exponential local moment (see Definition \ref{bnb}). \end{defi}
The following lemma explains why this is a well defined notion, i.e. the existence of exponential moment is independent of the embedding ``$i$''.
\begin{lemme} Let $G\subset SL(V)$ be the $k$-points of a semi-simple algebraic group and $\rho$ a finite dimensional $k$-algebraic representation of $G$. If  $\mu$ has an exponential local moment then the image of $\mu$ under $\rho$ has also an exponential local moment. \label{expomementrep}\end{lemme}
\begin{proof}
Each matrix coefficient $(\rho(g))_{i,j}$ of $\rho(g)$, for $g\in G$, is a fixed polynomial in terms of the matrix coefficients of $g$. Since for the canonical norm, $||g||\geq 1$ for every $g\in G$, we see that there exists $C>0$ such that $||\rho(g)||\leq ||g||^C$ for every $g\in G$. This suffices to show the lemma. \end{proof}

\subsubsection{Comparison between (the A-components of) the Cartan and Iwasawa decompositions.}

 Estimating the asymptotic behavior of the components of $S_n$ in the KAK decomposition will be crucial for us. We will derive these estimations from their analogs for the KAN decomposition.
 The following proposition explains why it is legal to do so:

\begin{prop}[Comparison between KAK and KAN ] Almost surely there exists a compact subset $C$ of $G$ such that for every $n\in \N^*$, $A_n \widetilde{A_n}^{-1}$ belongs to $C$.
In particular, there exists a compact subset $D$ of $GL(V)$ such that
$\rho (A_n) \rho (\widetilde{A_n})^{-1}$  belongs to $D$.
\label{compar} \end{prop}
\begin{proof} Since the kernel of the adjoint representation is finite, it suffices to show that there exists a compact subset $E$ of $GL(\mathfrak{g})$ such that $Ad(A_n)Ad(\widetilde{A_n}^{-1})$  belongs to $E$. This is equivalent to show that almost surely $\frac{\alpha(A_n)}{\alpha(\widetilde{A_n})}$ is in a random compact subset of $k$ for every $\alpha\in \Pi$.\;
Indeed, we decompose $\alpha$ into fundamental weights: $\alpha = \prod_{\beta\in \Pi} {w_\beta^{n_\beta}}$; $n_\beta \in \Z$. Hence,
\begin{equation}\frac{\alpha(A_n)}{\alpha(\widetilde{A_n})} = \prod_{\beta\in \Pi} \left( {\frac{w_\beta (A_n)}{w_\beta (\widetilde{A_n})}}\right)^{n_\beta}   \label{esmein}\end{equation}
By Theorem \ref{emilia}, for each $\beta \in \Pi$,
there exists a representation $(\rho_\beta,V_\beta)$ of $G$ whose highest weight is
$w_\beta$ and highest weight space is a line, say $k\; x_\beta$. Fix a $(\rho_\beta,A,K)$-good norm on $V_\beta$. Corollary \ref{youssif} applied to the representation $\rho_\beta$ gives then:
  $$||\rho_\beta (S_n)||= |{w_\beta (A_n)}| \;\;;\;\;
   \frac{||\rho_\beta (S_n) x_\beta||}{||x_\beta||} = |{w_\beta (\widetilde{A_n})}|$$
 Then (\ref{esmein}) becomes then
\begin{equation}\big| \frac{\alpha(A_n)}{\alpha(\widetilde{A_n})} \big|=
\prod_{\beta\in \Pi} \left( {\frac{||\rho_\beta (S_n)||}{\frac{||\rho_\beta (S_n) x_\beta||}{||x_\beta||}}} \right)^{{n_\beta}}\label{rap1}\end{equation}
It suffices to control the terms where $n_\beta \geq 0$. Since $G$ is Zariski-connected, $\rho_\beta$ is
in fact strongly irreducible. By Zariski density, $\rho_\beta (\Gamma_\mu)$ also. Hence we can apply  Proposition \ref{laprop}:  $$\textrm{a.s.}\;\;\;\;\;\;\;Sup_{n\in \N^*}\; \frac{||\rho_\beta (S_n)||}{\frac{||\rho_\beta (S_n) x_\beta||}{||x_\beta||}} < \infty $$
 This is what we want to show.
  \end{proof}

A version of the latter proposition ``in expectation'' will be needed.
\begin{prop}[Comparison between KAK and KAN in expectation]\label{comparaison2}
Assume that $\mu$ has an exponential local moment (Definition \ref{moshader}).
For every $\gamma>0$, there exist $\epsilon(\gamma)>0$ and $n(\gamma) \in \N^*$ such that for $0<\epsilon< \epsilon(\gamma)$, $n>n(\gamma)$ and every $\alpha\in \Pi$:

\begin{equation}\E \left(\big|\frac{\alpha(A_n)}{\alpha(\widetilde{A_n})}\big|^\epsilon \right)\leq (1+\epsilon\gamma)^n \;\;\;\;;\;\;\;\; \E \left(\big|\frac{\alpha(\widetilde{A_n})}{\alpha({A_n})}\big|^\epsilon \right)\leq (1+\epsilon\gamma)^n  \label{kitir1}\end{equation}
Moreover,
\begin{equation}\label{utile}\E ( ||{\rho(A_n)}{\rho(\widetilde{A_n}}^{-1})||^\epsilon) \leq (1+\epsilon\gamma)^n \end{equation}\end{prop}

\begin{proof} Let $\epsilon>0$ and $\alpha\in \Pi$. Let $\beta_1,...,\beta_s$ be an order of the simple roots appearing in identity (\ref{rap1}).  Holder inequality (for $s$ maps) applied to the same identity  gives:
$$ \E \left( \big| \frac{\alpha(A_n)}{\alpha(\widetilde{A_n})} \big|^\epsilon \right)\leq  \prod_{i=1}^{s} \Big [ \E \big[ \left( {\frac{||\rho_{\beta_i} (S_n)||}{\frac{||\rho_{\beta_i} (S_n) x_{\beta_i}||}{||x_{\beta_i}||}}} \right)^{ \epsilon s n_{\beta_i}}\big]\Big] ^{\frac{1}{s}}$$

Terms with $n_{\beta_i} \leq  0$ are less or equal to one. Hence, it suffices to control the terms
 where $n_{\beta_i} >0$. Fix such $i\in \{1,...,s\}$ and let $\gamma>0$.  By Lemma \ref{expomementrep},
 the image of $\mu$ under $\rho_{\beta_i}$ has an exponential local moment.
  Moreover, as explained in the previous proposition, $G$ being Zariski-connected, $\rho_{\beta_i}$ is
   strongly irreducible. Consequently, we can apply Proposition \ref{largedeviations} which shows that
 $\E \big[ \left( {\frac{||\rho_{\beta_i} (S_n)||}{\frac{||\rho_{\beta_i} (S_n) x_{\beta_i}||}{||x_{\beta_i}||}}} \right)^{\epsilon s  n_{\beta_i}}\big] \leq (1+\gamma \epsilon)^n $. Hence  $\E \left( \big| \frac{\alpha(A_n)}{\alpha(\widetilde{A_n})} \big|^\epsilon \right) \leq (1+\gamma  \epsilon)^n$. In the same way, we show the inequality on the right hand side of (\ref{kitir1}).\\
In particular, for every non zero weight $\chi$ of $(\rho,V)$ different from $\chi_\rho$,
 $\E \left([\chi(A_n) /\chi(\widetilde{A_n}) ]^\epsilon\right) \leq (1+\gamma \epsilon)^n$.
 Indeed, this follows from the expression $\chi = \chi_1\;/ \; \prod_{\alpha \in \Pi}\;{\alpha^{s_\alpha}}$ with $s_\alpha \in \N$ and the Holder inequality applied to (\ref{kitir1}). For $\chi=\chi_\rho$, a similar inequality holds because $\chi_\rho(A_n)/\chi_\rho(\widetilde{A_n})=||S_n||/ ||S_n x||$ for some $(\rho,A,K)$-good norm and every $x\in V_{\chi_\rho}$. This proves (\ref{utile}).
 \end{proof}

The following theorem shows that the ratio between the first two components in the Iwasawa decomposition is exponentially
small.
\begin{theo} [Exponential contraction in $KAN$] \label{iwa} Assume that $\mu$ has an exponential local moment and that
$\rho(\Gamma_\mu)$ is  contracting. Then there exists $\lambda>0$, such that for every $\epsilon>0$ small enough and all $n$ large enough:
 $$\E(|\frac{\widetilde{a_i(n)}}{\widetilde{a_1(n)}}|^\epsilon)\; \leq (1-\lambda \epsilon)^n \;\;\;\;\;;\;\;\;\;i=2,...,d$$

 We recall that $\widetilde{A_n}$ is the $A$-component of $S_n$ in the Iwasawa decomposition of $S_n$ in $G$ and that $\widetilde{a_1(n)},...,\widetilde{a_d(n)}$ are the diagonal components of $\rho(\widetilde{A_n})$ in the basis of weights.
 \end{theo}
\begin{remarque} When $k=\R$, no contraction assumption is needed.
Indeed, by a theorem of Goldsheild-Margulis \cite{Margulis}, a semigroup $\Gamma$ of $GL_d(\R)$ is strongly
irreducible and contracting if and only if its Zariski closure is.
Hence $\rho(\Gamma_\mu)$ is contracting if and only if $\rho(G)$ is. But $G$ is $\R$-split, hence the highest weight space of $\rho$ is a line, thus $\rho$ is contracting. \end{remarque}

Before proving the proposition, we state a standard lemma in this theory:
\begin{lemme}\cite{dekk} Let $G$ be a group,  $X$ be a $G$-space, $(X_n)_{n\in \N^*}$ a sequence of independent elements of $G$ with distribution $\mu$ and $s$ an additive cocycle on $G\times X$. Suppose that $\nu$ is a $\mu$-invariant probability measure on $X$ such that:
\begin{enumerate}
\item  $ \iint {s^+ (g,x) d\mu(g) d\nu(x)}< \infty $ where $y^+=\sup(0,y)$ for every $y\in \R$.
\item For $\p \otimes \nu$-almost every  $(\omega,x)$, \; $\lim_{n \rightarrow \infty} {s \left(X_n(\omega)...X_1(\omega),x\right ) }=+ \infty $. \end{enumerate}
Then $s$ is in $L^1 (\p \otimes \nu)$ and $\iint{s(g,x) d\mu (g) d\nu(x)}>0$\label{lemme}\end{lemme}

\textbf{Proof of Theorem \ref{iwa}: }
Since $\rho$ is contracting, $V_{\chi_\rho}$ is a line. Indeed, if $\{\eta_n; n\in \N\}$ is a sequence in $G$ such that $\{\rho(\eta_n);n\in \N\}$ is contracting then it is easy to see that
$\{\rho\left( a (\eta_n) \right); n \in \N\}$ is also contracting. $V_{\chi_\rho}$ is then a one dimensional subspace.
Therefore,
for some $\alpha \in \Theta_\rho$, $\frac{\widetilde{a_2(n)}}{\widetilde{a_1(n)}} = \frac{1}{\alpha(\widetilde{A_n})}$ and in general for $i\in \{2,...,d\}$, $\frac{\widetilde{a_i(n)}}{\widetilde{a_1(n)}}$ is of the form $1/ \prod_{\beta\in \Theta_\rho}{\beta^{m_\beta}(\widetilde{A_n})}$ with $m_\beta\in \N$ for every $\beta\in \Pi$. By Holder inequality, it suffices to treat the case where $\widetilde{a_i(n)}/\widetilde{a_1(n)}=1/\alpha(\widetilde{A_n})$ for some $\alpha\in \Theta_\rho$.
As in Proposition \ref{compar}, we decompose $\alpha$ into fundamental weights: $\alpha = \prod_{i=1}^s {w_{\beta_i} ^{n_{\beta_i}}}$, with $s\in \N^*$, $n_{\beta_i}\in \Z$ for every $i=1,...,s$. We denote $(\rho_{\beta_i},V_{\beta_i})$ the fundamental representation associated to $w_{\beta_i}$.
Using (\ref{rap1})
of the same proposition, we get for every $i=1,...,s$  a  $(\rho_{\beta_i},A,K)$-good norm on $V_i$ such that:
$$|\frac{\widetilde{a_i(n)}}{\widetilde{a_1(n)}}|^\epsilon = \Big[\prod _{i=1}^s\;\frac{||\rho_{\beta_i} (S_n)x_{\beta_i}||}{||x_{\beta_i}||}\Big]^{-\epsilon\; n_{\beta_i} }\;\leq Sup_{x\in X}\;{exp(-\epsilon s(S_n, x))}$$
where $X=\oplus_{i=1}^s P(V_{\beta_i})$ and $s$ is the cocycle defined on $G\times X$  by:
 $$s\left(g,([x_1],...,[x_s])\right)= \;\sum_{i=1}^s {n_{\beta_i} \log \frac{||\rho_{\beta_i} (g).x_i||}
 {{ ||x_i||}}}$$
 To apply Lemma \ref{cocycle}, we must verify that for some $\tau>0$, \begin{equation}\E\left(  exp(\tau sup_{x\in X}\;|s(X_1,x)| ) \right)< \infty \label{moment}\end{equation} and \begin{equation}\label{limneg}\lim \frac{1}{n} Sup_{x\in X}\; \E (-s (S_n,x)) < 0\end{equation}
By Lemma \ref{expomementrep}, there exists $\tau>0$ such that for every
 $i=1,...,s$,  $\E \left( ||\rho_{\beta_i} (X_1) ||^\tau \right) < \infty $.
 Holder inequality applied recursively ends the proof of (\ref{moment}).
  Now we concentrate on proving (\ref{limneg}).
Since $P(V_{\beta_i})$ is compact for every $i=1,...,s$, it suffices to show that for all sequences $\{x_{1,n};n \geq 0\}$,..., $\{x_{s,n};n \geq 0\}$ converging to non zero elements of $V_{\beta_1},...,V_{\beta_s}$:
$$\lim_{n\rightarrow \infty}{\frac{1}{n}{s\left(S_n,([x_{1,n}],...,[x_{s,n}])\right)}}=\;\lim_{n\rightarrow \infty} \frac{1}{n} \sum_{i=1}^s{{n_{\beta_i} \E \left( \log \frac{||\rho_{\beta_i} (S_n)x_{i,n}||}{ ||x_{i,n}||}\right)}}>0$$
Fix such sequences $\{x_{1,n};n\geq 0\}$,..., $\{x_{s,n};n \geq 0\}$. By Corollary \ref{coro} the limit above exists and is independent of
the sequences taken. Indeed, it is equal to the sum of the corresponding Lyapunov exponents. Denote by $L$ this limit. Fix a $\mu$-invariant probability measure $\nu$ on $X$, which exists by compactness of $X$. Again by Corollary \ref{coro},
$$L=\lim_{n\rightarrow\infty}{\frac{1}{n}
s\left(S_n(\omega),x\right)}=\;\lim_{n\rightarrow\infty} \frac{1}{n} \sum_{i=1}^s{{n_{\beta_i} \; \log \frac{||\rho_{\beta_i} \left(S_n(\omega)\right)x_{i}||}{ ||x_{i}||}}}\;\;\; \textrm{for $\p \otimes \nu$ - almost all $(\omega,x)$}$$
Consider the dynamical system $E=G^\N\times X$, the distribution $\eta=\p \otimes \nu$ on $E$,
  the shift $\theta: E \rightarrow E$,  $\left((g_0,......),x\right) \longmapsto \left((g_1,......),g_0.x\right)$.
   Since $\nu$ is $\mu$-invariant, $\eta$ is $\theta$-invariant. We extend the definition domain of $s$ from $G\times X$ to $G^\N \times X$ by setting $s(\omega,x):=s(g_0,x)$ if $\omega=(g_0,....)$.
   Since $\mu$ has an exponential moment, $s\in L_1(\eta)$. In consequence, we can apply the ergodic theorem (see \cite[Theorem 6.21]{brei}) which shows that $\frac{1}{n}{\sum_{i=0}^n {s\circ \theta^i (\omega,x)}}$ converges for $\eta$-almost every $(\omega,x)$ to a random variable $Y$ whose expectation is $\iint{s(g,x)d\mu(g)d\nu(x)}$.  Since $s$ is a cocycle,\;
    $s\left(S_n(\omega),x\right)={\sum_{i=0}^n {s\circ \theta^i (\omega,x)}}$. Hence,
   $$ \lim_{n\rightarrow\infty}{\frac{1}{n}
s\left(S_n(\omega),x\right)} = Y\;\;\;\;\;;\;\;\;\; \E_\eta(Y)=\iint{s(g,x)d\mu(g)d\nu(x)}$$
But  we have shown above that $Y$ is almost surely constant,because it is the sum of the corresponding Lyapunov exponents, and that it equal to $L$. Hence,
    $$L=  \iint{s(g,x)d\mu(g)d\nu(x)}$$
$L$ is positive if conditions (1) and (2) of Lemma \ref{lemme} are fulfilled.  Since $\mu$ has a moment of order one,
 condition (1) is readily satisfied.\\
Condition (2): we must verify that for $\p \otimes \nu$-almost all $(\omega,x)$,
 \begin{equation}s\left(S_n(\omega),x\right)=\sum_{i=1}^s{{n_{\beta_i} \log \frac{||\rho_{\beta_i} \left(S_n(\omega)\right) x_{i}||}{ ||x_{i}||}}\underset{n\rightarrow \infty}{\longrightarrow}}+\infty \label{eqqq}\end{equation}
  By Proposition \ref{laprop},
the $\p \otimes \nu$-almost
everywhere behavior at infinity of  $s\left( S_n(\omega),x\right)$ is the same as the $\p$-almost everywhere behavior of:
 $$\sum_{i=1}^s{{n_{\beta_i} \;\log ||\rho_{\beta_i} (S_n)||}}=\;\log \big|\alpha(A_n)\big|$$
The last equality  follows from the  expression of $\alpha$ in terms of the fundamental weights and  from Corollary
 \ref{youssif}.  Hence, we reduced the problem to proving that
 \;$|\alpha(A_n)| \underset{n\rightarrow\infty}{\overset{a.s}{\longrightarrow}}+\infty $ for every $\alpha \in \Theta_\rho$. \\
 $\rho(\Gamma_\mu)$ is strongly irreducible because $\Gamma_\mu$ is Zariski dense in $G$, $\rho$ is an irreducible representation of $G$ and $G$ is connected. Since by the hypothesis $\rho(\Gamma_\mu)$ is contracting, we can apply Theorem \ref{normalise}:\\
  $||.||$ being $(\rho,A,K)$-good norm, $|a_1(n)|=||\rho(S_n)||$.
  Hence a.s. every limit point of  $\frac{\rho(S_n)}{a_1(n)}$ is a rank one
  matrix. In particular, $\frac{a_2(n)}{a_1(n)},...,\frac{a_d(n)}{a_1(n)}$ converge a.s.  to zero.
  Equivalently, for  every weight $\chi\neq \chi_\rho$ of $V$, $|\chi_\rho(A_n) \;/\;\chi(A_n)|$ tends a.s. to
   infinity. From the expression of $\chi$ in terms of $\chi_\rho$, this is equivalent to say that for every
    $\alpha \in \Theta_\rho$, $|\alpha(A_n)|$ tends to infinity.


$$\Box$$
The following theorem shows that the ratio between the first two components in the Cartan decomposition is exponentially
small.
\begin{theo} [Exponential contraction in $KAK$]\label{iwabis}
 With the same hypotheses as in Theorem \ref{iwa}, there exists $\lambda>0$ such that for all $\epsilon>0$:
 $$\limsup_{n \rightarrow \infty} \big[\E (|\frac{a_i(n)}{a_1(n)}|^\epsilon) \big]^{\frac{1}{n}}< 1-\lambda\epsilon\;\;\;\;\;;\;\;\;\;\;i=2,...,d$$
\end{theo}
\begin{proof} Let $i\in \{2,...,d\}$. Since $|a_i\left(\rho(a)\right)|\leq |a_1\left(\rho(a)\right)|$ for every $a\in A^+$, it suffices to show the theorem for all $\epsilon>0$ small enough.  Write
$$\frac{a_i(n)}{a_1(n)}=\frac{a_i(n)}{\widetilde{a_i(n)}}\;\times\;\frac{\widetilde{a_1(n)}}{{a_1(n)}}\;\times\;\frac{\widetilde{a_i(n)}}{\widetilde{a_1(n)}}$$
Fix $\gamma>0$. By Propositions \ref{comparaison2} and \ref{iwa} and Holder inequality, we have for some $\lambda>0$, every $0<\epsilon<Min\{\epsilon(\gamma);\frac{1}{3\lambda}\}$ and $n>n(\gamma)$:
\begin{eqnarray}\E (|\frac{a_i(n)}{a_1(n)}|^\epsilon) \;\leq\; {(1+3\gamma\epsilon)}^{\frac{1}{3}} {(1+3\gamma \epsilon)}^{\frac{1}{3}} (1-3\lambda\epsilon)^{\frac{1}{3}}\;\leq\; (1+\gamma\epsilon)^2 (1-\lambda\epsilon) \;\leq\; 2(1+\gamma^2 \epsilon)(1-\lambda\epsilon)\nonumber\end{eqnarray}
We have used the inequality $(1+x)^r \leq 1+rx$ true for every $x\geq -1$ and $r\in ]0,1[$ and the inequality $(x+y)^2\leq 2(x^2+y^2)$ true for every $x,y\in \R$.
 It suffices to choose $\gamma=\frac{\sqrt{\lambda}}{2}$ for instance.
\end{proof}

We will see in Section \ref{subsequi} that in order to work with non Zariski-connected algebraic groups, it is convenient to work with the Cartan decomposition of the ambient group $SL_d(k)$ (see Section \ref{uff}). The following corollary will be useful. It is the analog of Theorem \ref{iwabis} for the KAK decomposition in $SL_d(k)$ (rather than in $G$).

\begin{corollaire}[Ratio in the $A$-component for the KAK decomposition of $SL_d(k)$]\label{hardini} For $g\in SL_d(k)$, we denote by $g=\widehat{k(g)} \widehat{a(g)} \widehat{u(g)}$ an arbitrary but fixed Cartan decomposition of $g$ in $SL_d(k)$ as described in Section \ref{uff}. We write $\widehat{a(g)}=diag\left(\widehat{a_1(g)},...,\widehat{a_d(g)} \right) $ in the canonical basis of $k^d$. With this notations and with the same assumptions as in Theorem \ref{iwa}, we have for some $\lambda>0$ and every $\epsilon>0$,

$$\limsup_{n\rightarrow \infty} \Big[{\E \left(\big|\frac{\widehat{a_i\left(\rho(S_n) \right)}}{\widehat{a_1\left(\rho(S_n) \right)}} \big|^\epsilon\right) } \Big]^{\frac{1}{n}} \leq 1-\gamma \epsilon \;\;\;\;\;;\;\;\;\;\; i=2,...,d$$
\end{corollaire}

 \begin{proof}  To simplify notations we omit $\rho$, so that $G$ is seen as a linear algebraic subgroup of $SL_d(k)$.
 Let $S_n=K_nA_nU_n$ be the Cartan decomposition of $S_n$ in $G$ (Section \ref{preliminaries}) and
  $S_n=\widehat{K_n}\widehat{A_n}\widehat{U_n}$ its Cartan decomposition in $SL_d(k)$ (Section \ref{uff}).
  Recall that $A_n$ is a diagonal matrix  $diag\left(a_1(n),...,a_d(n) \right)$ in the basis of weights
  while  $\widehat{A_n}$ is a diagonal matrix $diag\left(\widehat{a_1(n)},...,\widehat{a_d(n)} \right)$
  in the canonical basis of $k^d$. We will use the canonical basis and norm of $k^d$ (Section \ref{gene}).


 Theorem \ref{iwabis} shows that for some $\lambda>0$, every $\epsilon>0$ and all large $n$,
\begin{equation} \E \left(\big|\frac{{a_i}(n)}{{a_1}(n)} \big|^\epsilon \right) \leq (1-\gamma \epsilon)^n \;\;\;;\;\;\;i=2,...,d\label{veronica}\end{equation}
Since $K_n,\widehat{K_n}$ belong to compact subgroups in both decompositions, there exist $C_1,C_2>0$ such that for every $n$: $C_2||{A_n}||\leq||\widehat {A_n}||\leq C_1||{A_n}||$  and  $C_2||\bigwedge^2
{A_n}||\leq||\bigwedge^2 \widehat{A_n}||\leq C_1||\bigwedge^2 {A_n}||$.\\


\noindent By the definition of the  KAK decomposition in $SL_d(k)$,  we have a.s.
 $||\widehat{A_n}||=|\widehat{a_1(n)}|$ and $||\bigwedge^2 \widehat{A_n}||=|\widehat{a_1(n)}\widehat{a_2(n)}|$. For KAK in $G$,
  there exists a constant $C_3>0$ such that:
   $\frac{1}{C_3}|{a_1(n)}|\leq||{A_n}||\leq C_3|{a_1(n)}|$ and for $\p$-almost every $\omega$ there exists $i(\omega)\in \{2,...,l\}$ such that: $$\frac{1}{C_3}|{a_1(n)}{a_{i(\omega)}(n)}|\leq\;||\bigwedge^2{A_n(\omega)}||\leq\; C_3|{a_1(n)}{a_{i(\omega)}(n)}|$$
     Hence
$$ \E \left(\big|\frac{\widehat{a_2(n)}}{\widehat{a_1(n)}} \big|^\epsilon \right) =\E \big[\left(\frac{||\bigwedge^2 S_n||}{||S_n||^2}\right)^\epsilon\big]\;= \E \big[\left(\frac{||\bigwedge^2 \widehat{A_n}||}{||\widehat{A_n}||^2}\right)^\epsilon\big]\;\leq(C_1 C_3^3/C_2^2)^\epsilon\;\sum_{i=2}^d {E \left(\big|\frac{{a_i}(n)}{{a_1}(n)}\big|^\epsilon\right)} $$
By (\ref{veronica}), this is less or equal than $constant \times (1-\gamma\epsilon)^n$.
Since $|\widehat{a_2(g)}|\geq |\widehat{a_i(g)}|$ for $i>2$ and every $g\in SL_d(k)$, the proof is complete. \end{proof}

\subsubsection{Exponential convergence and asymptotic independence in KAK}
 We recall that the norm on $V$ we are working with is $(\rho,A,K)$-good (it is the one given by Theorem \ref{emilia}).
We recall also that the direct sum $V=\oplus_\chi V_\chi$ is good. When $k$ is archimedean, this norm is induced by a scalar product so that we can choose an orthonormal basis in each
 $V_\chi$. Let $(e_1,...,e_d)$ be the corresponding basis of $V$, $e_1$ is in particular a highest weight vector.
  Then, the norm on $V$ becomes $||x||^2=\sum_{i=1}^d{|x_i|^2}$, $x=\sum_{i=1}^d{x_ie_i}\in V$.
  When $k$ is non archimedean, one can choose a basis in each $V_\chi$ such that the norm induced becomes the Max norm. If $(e_1,...,e_d)$ is the corresponding basis of $V$, then $||x||= Max\{||x_i||; i=1,...,d\}$ for every
  $x=\sum_{i=1}{x_ie_i}\in V$. \smallskip

 Let $\rho^*: G \longrightarrow GL(V^*)$ be the contragredient representation of $G$ on $V^*$, that is  $\rho^*(g)(f)(x)=f\left( \rho(g^{-1})x \right)$ for every $g\in G$, $f\in V^*$, $x\in V$. For $g\in G$ and $f\in V^*$, $g.f$ will simply refer to $\rho^*(g)(f)$.
 Consider the norm operator on $V^*$, it is easy to see that it is $(\rho^*,A,K)$-good.
As explained in the preliminaries, $||.||$ induces a distance $\delta(.,.)$ on the projective space $P(V)$. The same holds for $P(V^*)$.\smallskip

Finally we recall the following notations: $M_n=X_1...X_n$, $S_n=X_n...X_1$ where ${X_i; i\geq 1}$ are independent random variables of law $\mu$. The KAK decomposition of $S_n$ in $G$ is denoted by $S_n=K_nA_nU_n$ with $K_n,U_n\in K$ and $A_n\in A^+$ (we have fixed a privileged way to construct the Cartan decomposition). We write $\rho(A_n)=diag\left(a_1(n),...,a_d(n) \right)$ in the basis of weights.  When it comes to the random walk $\{M_n;n\in \N^*\}$ we simply write $M_n=k(M_n)a(M_n)u(M_n)$ its KAK decomposition.

\begin{theo} [Exponential convergence in KAK]
Suppose that $\mu$ has an exponential local moment and that $\rho(\Gamma_\mu)$ is contracting.
Denote by $x_\rho$ a highest
weight vector ($e_1$ for example), then for all $\epsilon>0$:
 $$\limsup_{n \rightarrow \infty} \big[\E (\delta^\epsilon (k(M_n)[x_\rho],Z_1)) \big]^{\frac{1}{n}}< 1\;\;;\;\;
 \limsup_{n \rightarrow \infty} \big[\E (\delta^\epsilon (U_n^{-1}.[x_{\rho}^*],Z_2)) \big]^{\frac{1}{n}}< 1$$
 \label{expokak}
 where $Z_1$ (resp. $Z_2$) is a random variable on $P(V)$ (resp. $P(V^*)$)  with law $\nu$ (resp. $\nu^*$) -the unique $\mu$ (resp. $\mu^{-1}$) -invariant probability measure. \end{theo}
 \begin{remarque}
From the previous theorem, we deduce by applying the Borel Cantelli lemma that $k(M_n)[x_\rho]$ converges almost surely while  $K_n[x_\rho]=k(S_n)[x_\rho]$ converges only in law. This can  also be directly derived from Corollary \ref{essentiell}. \end{remarque}
\begin{proof}
 For simplicity, we write $S_n$, $K_n$,$A_n$,$U_n$ instead of $\rho(S_n)$, $\rho(A_n)$,
 $\rho(U_n)$. By the canonical identification between $V$ and $(V^*)^*$, ${(e_1^*)}^*$ will refer to $e_1$. Let $Z\in P(V^*)$ be the almost sure limit of $S_n^{-1}.[f]$, for every $[f]\in P(V^*)$,   obtained by Theorem \ref{direction}. Since for every $i=1,...,d$, $A_n^{-1}.e_i^*=a_i(n)e_i^*$
 and $S_n=K_nA_nU_n$, we have for every $f\in V^*$ of norm one, such that
$e_1(K_n^{-1}.f) \neq 0$,

$$S_n^{-1}.f= e_1 (K_n^{-1}.f) \;a_1(n) \;U_n^{-1}.e_1^* \;+ \sum_{i=2}^d {O(a_i (n))}$$
$$U_n^{-1}.e_1^* = \frac{1}{e_1 (K_n^{-1}.f)}\frac{S_n^{-1}.f}{a_1(n)}
\;+ \; \frac{1}{e_1 (K_n^{-1}.f) } \sum_{i=2}^d {O(\frac{a_i (n)}{a_1(n)})}$$
Recall that $\delta([x],[y])=\frac{||x \wedge y ||}{||x||||y||}$; $[x],[y]\in P(V^*)$. Hence
 $$\delta (U_n^{-1}.[e_1^*],Z)\leq \frac{1}{|e_1(K_n^{-1}.f)|} \left(\frac{||S_n^{-1}.f||}
{|a_1(n)|}\;\delta(S_n^{-1}.[f],Z) +\sum_{i=2}^d {O(|\frac{a_i (n)}{a_1(n)}|)} \right)$$
Since $|a_1(n)|=||S_n||$  and $||f||=1$, \; $||S_n^{-1}.f||=Sup_{||x||=1}\;{|f(S_nx)|}\;\leq |a_1(n)|$.
Hence
 \begin{equation}\delta (U_n^{-1}.[e_1^*],Z) \leq   \frac{1}{|e_1(K_n^{-1}.f)|}\left(
 \delta(S_n^{-1}.[f],Z) +\sum_{i=2}^d O(|\frac{a_i(n)}{a_1(n)}|) \right) \label{lll}
\end{equation}
Let $C(k)=\frac{1}{\sqrt{d}}$ (resp. $C(k)=1$) when $k$ is archimedean (resp. non archimedean). The choice
 of the norm on $V$ implies that a.s. there exists $i=i(n,\omega)\in\{1,...,d\}$, such that
$|e_1(K_n^{-1}.e_{i}^*)| \geq C(k)$. Indeed, in the non archimedean case,
$1=||K_n.e_1||= Max\{|K_n.e_1(e_i^*)|; i=1,...,d\}$. Hence for some random $i=i(n,\omega)$, $|e_1(K_n^{-1}.e_i^*)|=|K_n.e_1(e_i^*)|= 1$ and in the archimedean case,
$1=||K_n.e_1||= \sum_{i=1}^d {|K_n.e_1 (e_i^*)|^2}=\sum_{i=1}^d{|e_1(K_n^{-1}.e_i^*)|^2}$.
Hence one can write for every $\epsilon>0$:
 \begin{equation}\label{chita}\E (\delta^\epsilon (U_n^{-1}.[e_1^*],Z)) \leq \sum_{i=1}^d {\E \left(\delta^\epsilon (U_n^{-1}. [e_1^*],Z)\;;\; \mathds{1}_{|e_1 (K_n^{-1}
. e_{i}^*)| \geq C(k)}\right)}\end{equation}
In (\ref{chita}), for every $i=1,...,d$, on the event ``$ |e_1 (K_n^{-1}
. e_{i}^*)| \geq C(k) $'', we apply (\ref{lll}) with $f=e_i$. Since $\epsilon>0$ can be taken smaller than one, $C(k)^\epsilon\geq C(k)$ and $(x+y)^\epsilon\leq x^\epsilon +y^\epsilon$ for every $x,y\in {\R}_{+}$.
 We get then:
\begin{equation}\E (\delta^\epsilon (U_n^{-1}.[e_1^*],Z)) \leq \frac{1}{C(k)}\sum_{i=1}^d \E (\delta^{\epsilon} (S_n^{-1}. [e_i^*],Z))+ \frac{1}{C(k)}{ \sum_{i=2}^d\E(|\frac{a_i(n)}{a_1(n)}|^\epsilon)}\end{equation}

\noindent Theorem \ref{iwabis} shows that: $\E (|\frac{a_i(n)}{a_1(n)}|^\epsilon)$ is sub-exponential for $i=2,...,d$.\\
Theorem \ref{direction} shows that for  every $i=1,...,d$, $ \E (\delta^{\epsilon} (S_n^{-1}.
 [e_i^*],Z))$ is sub-exponential. In the same way, we show the exponential convergence of $k(M_n)[x_\rho]$.
\end{proof}

We have shown that  $U_n^{-1}.[x_\rho^*]$ converges a.s. and $K_n[x_\rho]$ in law. In the following theorem, we show that these two variables become independent at infinity, with exponential ``speed''. This is Theorem \ref{marie-therese} from the introduction. We recall its statement.
\begin{theo}[Asymptotic independence in the KAK decomposition]

With the same assumptions as in Theorem \ref{expokak}, there exist \textbf{independent random variables} $Z\in P(V^*)$ and $T\in P(V)$  such that for every $\epsilon>0$,  every $\epsilon$-holder (real) function $\phi$ on $P(V^*)\times P(V)$ and all large $n$:
$$\big|\E\left(\phi([U_n^{-1}.x_{\rho}^*],[K_nx_{\rho}]) \right) - \E \left( \phi(Z,T) \right) \big|\leq ||\phi||_{\epsilon} \rho(\epsilon)^n $$
where $$||\phi||_\epsilon= Sup_{[x],[y],[x'],[y']}\;{\frac{\big|\phi([x],[x'])-\phi([y],[y'])\big|}{\delta^{\epsilon}([x],[y])+\delta^{\epsilon}([x'],[y'])}}$$
\label{independence}\end{theo}

\begin{proof}
Let $\epsilon>0$. The analog of Theorem \ref{expokak} for $U_n^{-1}.[x_\rho^*]$ does not hold for $K_n[x_\rho]$ because it converges  only in law.
  However, we have the following nice estimate:
for some $\rho(\epsilon)\in ]0,1[$ and all $n$ large enough:
\begin{equation}\label{original}\E\big[\delta^{\epsilon}\left(K_n[x_\rho] \;,\; k(X_n...X_{\lfloor \frac{n}{2} \rfloor})[x_\rho]\right) \big]\leq \rho(\epsilon)^n\end{equation}
 Indeed, by independence ${(X_1,...,X_n)}$ has the same law as ${(X_n,...,X_1)}$ for every $n\in \N^*$. Therefore, for every $n\in \N^*$: $$\E\big[\delta^{\epsilon} \left(K_n[x_\rho]\;,\; k(X_n...X_{\lfloor \frac{n}{2} \rfloor})[x_\rho] \right) \big]= \E\big[\delta^{\epsilon} \left(k(M_n)[x_\rho] \;,\;k(M_{n-\lfloor \frac{n}{2} \rfloor +1})[x_\rho] \right)\big]$$
 It suffices now to apply twice the first convergence of Theorem \ref{expokak} and the triangle inequality.\\


 Now let $\phi$ be an $\epsilon$-holder function on $P(V^*) \times P(V)$, $(X'_n)_{n\in \N}$ increments with law $\mu$ independent from $(X_n)_{n\in \N}$. We similarly write $M'_n=X'_1...X'_n$. \\Let $Z=\lim U_n^*[x_\rho]$ and $T=\lim k(M'_n) [x_\rho]$ (a.s. limits given by Theorem \ref{expokak}). \textbf{The random variables $T$ and $Z$ are in particular independent}.   We write:
 $$E\left(\phi(U_n^{-1}.[x_\rho^*],K_n[x_\rho]) \right) -\E\left(\phi(Z,T) \right)= \Delta_1 + \Delta_2 + \Delta_3+\Delta_4$$
 where $$\Delta_1=\E\left(\phi(U_n^{-1}.[x_\rho^*],K_n[x_\rho]) \right) - \E\left(\phi(U_{\lfloor \frac{n}{2} \rfloor}^{-1}.[x_\rho^*],K_n[x_\rho]) \right)$$ $$\Delta_2=\E\left(\phi(U_{\lfloor \frac{n}{2}  \rfloor}^{-1}.[x_\rho^*],K_n[x_\rho]) \right)- \E\left(\phi(U_{\lfloor \frac{n}{2} \rfloor}^{-1}.[x_\rho^*],k(X_n...X_{\lfloor \frac{n}{2}  \rfloor +1})[x_\rho]) \right) $$
  \begin{equation}\Delta_3= \E\left(\phi(U_{\lfloor \frac{n}{2} \rfloor}^{-1}.[x_\rho^*],k(M'_{n-\lfloor \frac{n}{2}  \rfloor} ).[x_\rho] ) \right) - \E\left(\phi(Z,k(M'_{n-\lfloor \frac{n}{2}  \rfloor} )[x_\rho] \right)  \nonumber\end{equation}
  $$\Delta_4=\E\left(\phi(Z,k(M'_{n-\lfloor \frac{n}{2}  \rfloor} )[x_\rho]) \right) - \E\left(\phi(Z,T) \right)  $$

    In $\Delta_3$, we have replaced $k(X_n...X_{\lfloor \frac{n}{2}  \rfloor +1})$  with $ k(M'_{n-\lfloor \frac{n}{2}  \rfloor})$ because, on the one hand they have the same law and on the other hand, the processes $k(X_n...X_{\lfloor \frac{n}{2}  \rfloor +1})$ and $U_{\lfloor \frac{n}{2} \rfloor}$ that appear in the last term of the right hand side of $\Delta_2$ are independent. \\

  $\bullet$ By Theorem \ref{expokak}, there exist $\rho_1(\epsilon),\rho_2(\epsilon) \in ]0,1[$ such that: $|\Delta_1| \preceq ||\phi||_\epsilon\;\rho_1(\epsilon)^n+||\phi||_\epsilon\;\rho_1(\epsilon)^{\frac{n}{2}}$;\; $|\Delta_3| \preceq ||\phi||_\epsilon\;\rho_1(\epsilon)^{\frac{n}{2}}$ \; and \;$|\Delta_4|\preceq ||\phi||_\epsilon \;\rho_2(\epsilon)^{\frac{n}{2}}$.\\

  $\bullet$ By (\ref{original}), $\Delta_2 \preceq ||\phi||_\epsilon\; \rho_3(\epsilon)^n$.\\

\end{proof}

\subsection{Estimates in the Cartan decomposition - the non-connected case}
\label{subsequi} Recall that $k$ is a local field, $\mathbf{G}$  a $k$-algebraic group, $G$ its $k$-points which we assume to be $k$-split. We denote by $\mathbf{G}^0$ its Zariski-connected component which we assume to be semi-simple and by $G^0$ its $k$-points. Finally, $\rho$ is a $k$-rational representation of $G$ into some $SL_d(k)$. We write $V=k^d$ and $P(V)$ the projective space. \\
In other terms, we consider the same situation as in Section \ref{preliminaries} except that $\mathbf{G}$ \textbf{is no longer assumed connected},
  a fortiori $\rho(G)$.
 The $KAK$ and $KAN$ decompositions do not necessarily hold for the algebraic groups  $G$,
  $\rho(G)$ but  are valid for $G^0$ or $\rho(G^0)$. However, one can still use the KAK decomposition of the ambient group $SL(V)$.\\
  We use then the notations and conventions of Section \ref{uff} regarding the Cartan decomposition in $\mathbf{SL_d}$.
  We consider the canonical basis $(e_1,...,e_d)$ and canonical norm on $V=k^d$ (see Section \ref{gene}).
  For each $g\in SL_d(k)$, we denote by $g=k(g)a(g)u(g)$ an arbitrary but fixed
  Cartan decomposition in $SL_d(k)$ and write $a(g)=diag\left(a_1(g),...,a_d(g)\right)$. \\
  We consider a probability measure $\mu$  on $G$ such that $\Gamma_\mu$ is Zariski dense in $G$. As usual, we denote by  $S_n=X_n...X_1 $  the  right random walk. \\

   The  aim of this section is to prove that the main results of   Section \ref{subsestimate}
   hold for the Cartan decomposition in $SL_d(k)$ rather than merely in G.
   Our first task will be to prove the following theorem, which is the analog of Theorem
\ref{iwabis} for the $KAK$ decomposition in $SL_d(k)$.
\begin{theo} Assume that the representation $\rho|_{ G^0}$ is irreducible. Let $\mu$ be a probability measure on $G$ having
an exponential local moment (see Definition \ref{moshader}) and such that $\rho(\Gamma_\mu)$ is contracting.  Then for every $\epsilon>0$, $$\limsup_{n \rightarrow \infty}  \big[\E \left(\big|\frac{a_2\left(\rho(S_n)\right)}{a_1\left(\rho(S_n)\right)}\big|^\epsilon\right)\big]^{\frac{1}{n}} < 1$$   \label{ratiosld} \end{theo}
Our next task will be to adapt the proof of Theorem \ref{expokak} (exponential convergence in the KAK decomposition) and Theorem \ref{independence} (asymptotic independence in the KAK decomposition) to the Cartan decomposition of $SL_d(k)$. This can be done easily using  Theorem \ref{ratiosld}. Indeed it will be sufficient to replace
 $x_\rho$, highest weight of $\rho$,  with $e_1$ (which is the highest weight for the natural representation of $SL_d(k)$ on $k^d$) and $KAK$ in $G$ with $KAK$ in $SL_d(k)$. By writing the Cartan decomposition of $\rho(S_n)$ in $SL_d(k)$ as $\rho(S_n)=K_nA_nU_n$, we obtain:

\begin{theo} With the same assumptions as in Theorem \ref{ratiosld}, there exist random variables $Z_1\in P(V)$ and $Z_2\in P(V^*)$ such that
 $$\limsup_{n \rightarrow \infty} \big[\E (\delta^\epsilon (k(M_n)[e_1],Z_1)) \big]^{\frac{1}{n}}< 1\;\;;\;\;\limsup_{n \rightarrow \infty} \big[\E (\delta^\epsilon (U_n^{-1}.[e_1^*],Z_2)) \big]^{\frac{1}{n}}< 1$$
\label{direction1} \end{theo}
\begin{theo} With the same hypotheses as in Theorem \ref{ratiosld}, there exists \textbf{independent random variables} $Z\in P(V^*)$ and $T\in P(V)$, $\rho \in ]0,1[$, $n_0>0$ such that, for every $\epsilon>0$, every $\epsilon$-holder (real) function $\phi$ on $P(V^*)\times P(V)$, every $n>n_0$ we have:
$$\big|\E\left(\phi([U_n^{-1}.e_1^*],[K_ne_1]) \right) - \E \left( \phi(Z,T) \right)\big| \leq ||\phi||_{\epsilon} \rho^n $$
where $$||\phi||_\epsilon= Sup_{[x],[x'],[y],[y']}{\frac{|\phi([x],[x'])-\phi([y],[y'])|}{\delta^{\epsilon}([x],[y])+\delta^{\epsilon}([x'],[y'])}}$$
\label{independence1}\end{theo}

Before proving Theorem \ref{ratiosld}, we give some easy but important facts.\\

\begin{defi} Let $\tau= inf \{n\in \N^*;\; S_n \in G^0\}$ i.e. the  first time the random walk $(S_n)_{n\in \N^*}$ hits $G^0$. Recursively,  for every $n\in \N$,  $\tau(n+1)=inf\{k> \tau(n); S_{k}\in G^0\}$  \end{defi}
For every $n\in\N^*$, $\tau(n)$ is a.s. finite. Indeed, by the Markov property it suffices to show that $\tau$ is almost surely finite: let $\pi$ be the projection $G \rightarrow G/G^0$, $\tau$ is then the first time the \textbf{finite} states Markov chain $\pi(S_n)$ -it is in fact a random walk because $G^0$ is normal in $G$ - returns to identity.
\begin{lemme} If $\mu$ is a probability measure on $G$ with an exponential local moment (see Definition \ref{moshader}), then the distribution $\eta$ of $S_\tau$  also has an exponential local moment.\label{amar}
\end{lemme}
\begin{proof} We identify $G$ with a closed subgroup of $GL_r(k)$. For every $\alpha>0$:
\begin{equation}\E \left( ||S_\tau||^\alpha \right)=\sum_{k\in \N^*}\;{\E \left(||S_k||^\alpha\;;\;\mathds{1}_{\tau = k } \right)}\;\leq \sum_{k\in \N^*}\;\sqrt {\E (||S_k||^{2\alpha})}\;\sqrt{\p (\tau=k)}\label{kilmii}\end{equation}
where we used the Cauchy-Schwartz inequality on the right hand side.  Since $\mu$ has an exponential moment, there exists $\alpha_0>0$ such that: $1\leq \E (||X_1||^{2\alpha_0})=C< \infty $. Impose $\alpha < \alpha_0$. Since $x \mapsto x^{\frac{\alpha_0}{\alpha}}$ is convex,  the Jensen inequality gives:\;\;
 $\E (||X_1||^{2\alpha}) \leq \E (||X_1||^{2\alpha_0})^{\frac{\alpha}{\alpha_0}} = C^{\frac{\alpha}{\alpha_0}}$.
The norm being sub-multiplicative, we have by independence: $\E (||S_k||^{2\alpha})\leq \big[\E (||X_1||^{2\alpha}) \big]^k$  for every $k\in \N^*$.  Hence
\begin{equation}\E (||S_k||^{2\alpha}) \leq (C^{\frac{1}{\alpha_0}})^{\alpha k}\;\;\;; \;\;k\in \N^*\label{ghassan}\end{equation}
On the other hand, recall that $\tau$ is the first time the finite states Markov chain $\pi(S_n)$ returns to identity.
The Perron-Frobenius theorem implies that $\pi(S_n)$ becomes equidistributed exponentially fast so that $\p (\tau >k)$ is exponentially decaying. In particular, there exists a constant $\lambda>0$ such that \begin{equation}\p (\tau=k)\leq exp(-\lambda k )\label{badi3a}\end{equation}
Combining (\ref{kilmii}), (\ref{ghassan}) and (\ref{badi3a}) gives with $D=C^{\frac{1}{\alpha_0}}$: $$\E \left( ||S_\tau||^\alpha \right) \leq \sum_{k\in \N^*} {D^{k\alpha/2}\; exp(-\lambda k/2 )}$$
 It suffices to choose $\alpha>0$ small enough such that the latter sum is finite $(\alpha < \frac{\lambda}{\log (D)}$ works).\end{proof}
\begin{corollaire} Suppose that $\mu$ has an exponential local moment, $\rho |_ {G^0}$
is irreducible and $\rho(\Gamma_\mu)$ is contracting. Then for every $\epsilon>0$,  $$\limsup_{n \rightarrow \infty} \Big[\E \left(\;\big|\frac{{a_2\left(\rho(S_{\tau(n)})\right)}}{{a_1\left(\rho(S_{\tau(n)})\right)}}\big|^\epsilon\;\right) \Big]^{\frac{1}{n}}< 1$$ \label{coro1}\end{corollaire}
\begin{proof}
 The variables $\{\tau(i+1)- \tau(i);\;i \geq 1\}$ are independent and have the same law $\tau=\tau(1)$. Hence, the process $\left(S_{\tau(n)}\right)_{n\in \N^*}$ has the same law as the usual right random walk on $G^0$ associated to the probability measure $\eta$.\\
$\bullet$ First we show that $\Gamma_\eta$ is Zariski dense in $G^0$. We claim that $\Gamma_\eta = \Gamma_\mu \cap G^0$. Indeed, recall that $\Gamma_\eta$ is the smallest closed  semigroup (for the natural topology of $End_d(k)$ induced by that of $k$) in $G^0$ containing the support of $\eta$. Hence, $M\in \Gamma_\eta$ if and only if for every neighborhood $O$ of $M$ in $G^0$,  $\p (\exists n \in \N^*; S_{\tau(n)} \in O ) >0 $. On the other hand, $G^0$ is open in $G$ because $G/G^0$ is finite. Thus, $M\in \Gamma_\mu \cap G^0$ if and only if for every neighborhood $O$ of $M$ in $G^0$, $\p (\exists n \in \N^*; S_{n} \in O) >0$  or equivalently $\p (\exists n \in \N^*; S_{\tau(n)} \in O) >0$. This shows indeed that $\Gamma_\eta = \Gamma_\mu \cap G^0$.
 \\Since $\Gamma_\mu$ is Zariski-dense in $G$ and $G^0$ is Zariski-open in $G$, we deduce that $\Gamma_\eta$ is Zariski dense in $G^0$.\\
$\bullet$ Next, we show that $\rho(\Gamma_\eta)$ is contracting. Indeed, by Lemma \ref{proximal}, $\rho(\Gamma_\mu)$ has a proximal element, say $\rho(\gamma)$ with $\gamma\in \Gamma_\mu$,  then $\rho(\gamma)^{[G/G^0]}=\rho(\gamma^{[G/G^0]})$ is also proximal with $\gamma^{[G/G^0]}$ in $\Gamma_\mu \cap G^0=\Gamma_\eta$. Hence $\rho(\Gamma_\eta)$ is proximal whence, again by Lemma \ref{proximal}, contracting.\\
In consequence, we are in the following situation: $G^0$ is the group of $k$-points of a connected algebraic group  and $\eta$ is a probability measure on  $G^0$ such that the semigroup $\Gamma_\eta$ is Zariski dense in $G^0$. Moreover, by Lemma \ref{amar}, $\eta$ has an exponential local moment. Finally $\rho|_{G^0}$ is an irreducible representation of $G^0$ such that $\rho|_{G^0}(\Gamma_\eta)$ is contracting. An appeal to Corollary  \ref{hardini} ends the proof.
\end{proof}

\begin{lemme} Let $\ell=\E(\tau)$.\\
(i) The Lyapunov exponent associated to the random walk $\rho(S_{\tau(n)})$ (or in other terms to the distribution $\rho(\eta)$) is $\ell \lambda_1 $, where $\lambda_1$ is the first Lyapunov exponent associated to $\rho(S_n)$.\\
(ii)  For every $\epsilon>0$, there exist $\rho(\epsilon)\in ]0,1[$, $n(\epsilon)\in \N^*$ such that for $n>n(\epsilon)$:
 $$\p (|\frac{1}{n}\tau(n) - \ell | >\epsilon )\leq \rho(\epsilon)^{n}$$
\label{taun}\end{lemme}

\begin{proof}
The stopping time $\tau(n)$ is the sum of the independent, $\tau$-distributed random variables $\{\tau(i+1)-\tau(i);i\geq 1\}$. By the usual strong law of large numbers, a.s. $\lim \frac{\tau(n)}{n} = \ell$, so that, $\frac{1}{n} {\log ||S_{\tau(n)}||}=\frac{\log ||S_{\tau(n)}||}{\tau(n)} \times \frac{\tau(n)}{n}$ converges almost surely towards $\lambda_1 \ell$. Item (ii) is an application of a classical large deviation inequality for i.i.d sequences: Lemma \ref{largeiid} below. To apply the latter, we should check that for some $\xi>0$, $\E \left(exp(\xi \tau) \right)< \infty $. Indeed, by (\ref{badi3a}), there exists $\xi>0$ such that for every $y\in \R_+$: $\p (\tau > y ) \leq exp(-\xi y)$. Hence, for every $t>0$, write: $$\E\left(exp(t \tau) \right)=\int_{0}^\infty {\p \left(exp(t \tau) > x \right)\; dx} = 1+\;\int_{1}^\infty {\p \left(\tau > \frac{\log(x)}{t} \right)\; dx}\leq 1+ \int_{1}^\infty {exp(-\xi\frac{\log(x)}{t})\;dx}$$
 The latter is finite as soon as $t < \xi$.
\end{proof}
The following lemma is classical in the theory of large deviations and is a particular case of the well-known Cramer Theorem. One can see \cite{stroo}, Lemma 3.4 Chapter 3 for example.
 \begin{lemme} [Large deviations theorem for i.i.d. sequences]
 Let $(X_n)_{n\in \N}$ be a sequence of independent, identically distributed real random variables. If for some $\xi >0$, $\E \left(exp(\xi |X_1|) \right) < \infty$, there exists a positive function $\phi$ on $\R^*$ such that for every $\epsilon>0$:
 $$\p \left( |\frac{1}{n} \sum_{i=1}^n {X_i} - \E(X_1) | \geq \epsilon\right) \leq exp\left(-n\phi(\epsilon) \right)$$
Moreover, one can take $\phi(\epsilon)=Sup_{0<t<\xi}\{t\epsilon - \psi(t) \}$ where $\psi(t)=\log \Big(\E\big[exp\left(t (X_1-\E(X_1) ) \right) \big] \Big)$.
 \label{largeiid}\end{lemme}


 \textbf{Proof of Theorem \ref{ratiosld}:}
To simplify  notations we omit $\rho$, so that $G$ in seen as a subgroup of $SL_d(k)$.
 Let $N\in \N^*$, $\epsilon>0$, $0<\epsilon'<l$ to be chosen in terms of $\epsilon$. By definition of the $KAK$ decomposition in $SL_d(k)$, what we want to prove is that for all $\epsilon>0$ small enough
$$ \limsup_{N\rightarrow \infty}{\E \big[\left( \frac{||\bigwedge^2 S_N||}{||S_N||^2}  \right)^\epsilon\big]} < 1$$
 Let $n=\lfloor\frac{N}{\ell}\rfloor$, so that for $N\geq N_1(\epsilon')=\frac{l(l+\epsilon')}{\epsilon'}$, \; $n(l-\epsilon')\leq N\leq n(l +\epsilon')$. We wish to have $\tau(n)$ and $N$ in the same interval  with high probability. \\
 Let $A_n$ be the event ``$\{\tau(n) \in [n(l-\epsilon');n(l+\epsilon')]\}$''. By Lemma \ref{taun}, there exists $\rho(\epsilon')\in ]0,1[$ such that $\p (A_n)\geq 1-\rho(\epsilon')^n$.
 We have then:

$$\E \big[\left( \frac{||\bigwedge^2 S_N||}{||S_N||^2}  \right)^\epsilon\big]\;\leq\;\E \left( \frac{||\bigwedge^2 S_N||^\epsilon}{||S_N||^{2\epsilon}} \mathds{1}_{A_n}\right)\; +\; \p (\Omega \setminus A_n) \;
\leq \;\underset{(I)}{\underbrace{ \E \left( \frac{||\bigwedge^2 S_N||^\epsilon}{||S_N||^{2^\epsilon}}  \;\mathds{1}_{A_n}\right)}} + \rho(\epsilon')^n$$
The first inequality is due to the fact that  $\frac{||\bigwedge^2 S_N||^\epsilon}{||S_N||^{2\epsilon}}\leq 1$. Since $n\geq N/(l+\epsilon') \geq N/2l$,\; $\rho(\epsilon')^n \leq \left(\rho(\epsilon')^{\frac{1}{2l}}\right)^N$. Hence it suffices to estimate (I).
$$(I) \leq \underset{(II)}{\underbrace{\E \left(\frac{||\bigwedge^2(X_N...X_{\tau(n)+1}\;S_{\tau(n)})||^\epsilon}{||X_N...X_{\tau(n)+1}S_{\tau(n)}||^{2\epsilon}} \;\mathds{1}_{N\geq \tau(n);{A_n}} \right)}}+\underset{(III)}{\underbrace {{\E\left(\frac{||\bigwedge^2(X_{N+1}^{-1}...X_{\tau(n)}^{-1}\;S_{\tau(n)})||^{\epsilon}}{||X_{N+1}^{-1}...X_{\tau(n)}^{-1}\;S_{\tau(n)}||^{2\epsilon}} \;;\;\mathds{1}_{N< \tau(n);A_n} \right)}}}$$

 $(III)$ is treated similarly as $(II)$. Since $||\bigwedge^2 g ||\leq ||g||^2 $; $\frac{1}{||g||} \leq ||g^{-1}||$; $||g^{-1}||\leq ||g||^{d-1}$  for every $g\in SL_d(k)$, we have:
 $$(II) \leq \E \left((||X_N||...||X_{\tau(n)+1}||)^{2d\epsilon}\; \frac{||\bigwedge^2 S_{\tau(n)}||^\epsilon}{||S_{\tau(n)}||^{2\epsilon}}\;;\;\mathds{1}_{N\geq \tau(n);A_n} \right) $$

\begin{eqnarray}(II)^2 &\leq& \E \left((||X_N||...||X_{\tau(n)+1}||)^{4d\epsilon}\;;\;\mathds{1}_{N\geq \tau(n);\;A_n} \right)  \; \E \left(\frac{||\bigwedge^2 S_{\tau(n)}||^{2\epsilon}}{||S_{\tau(n)}||^{4\epsilon}} \right)\label{tesada2}\\
 &= & \sum_{k=0}^{\infty} {\E \left((||X_N||...||X_{k+1}||)^{4d\epsilon}\;;\;\mathds{1}_{N\geq k;\;A_n}\;\mathds{1}_{\tau(n)=k} \right)}\;\E \left(\frac{||\bigwedge^2 S_{\tau(n)}||^{2\epsilon}}{||S_{\tau(n)}||^{4\epsilon}} \right)\nonumber\\
&\leq & \sum_{k=n(l-\epsilon')}^{n(l+\epsilon')} {\E \left((||X_N||...||X_{k+1}||)^{4d\epsilon}\right)}\;\E \left(\frac{||\bigwedge^2 S_{\tau(n)}||^{2\epsilon}}{||S_{\tau(n)}||^{4\epsilon}}\right) \label{ex1}\\
&\leq& \sum_{k=n(l-\epsilon')}^{n(l+\epsilon')} {\Big[\E \left(||X_1||^{4d\epsilon}\right)\Big]}^{|N-k|}\;\E \left(\frac{||\bigwedge^2 S_{\tau(n)}||^{2\epsilon}}{||S_{\tau(n)}||^{4\epsilon}}\right)\label{ex2} \end{eqnarray}

The bound (\ref{tesada2}) is obtained by the Cauchy-Schwartz inequality,  (\ref{ex1}) follows from the fact that on the event $A_n$,\; $\tau(n)\in [n(l-\epsilon');n(l+\epsilon')]$. Finally (\ref{ex2}) is due to the sub-multiplicativity of the norm and the independence of $X_N,...,X_{k+1}$.\\
Since $\mu$ has an exponential local moment, for $\epsilon$ small enough, $1\leq \E \left(||X_1||^{4d\epsilon}\right)=C(\epsilon) < \infty $. Moreover,  $n(l-\epsilon')<N< n(l+\epsilon')$, hence  $\sum_{k=n(l-\epsilon')}^{n(l+\epsilon')} {\big[\E \left(||X_1||^{4d\epsilon}\right)\big]^{|N-k|} }\leq 2n\epsilon' C(\epsilon)^{2n\epsilon'}\leq C(\epsilon)^{3n\epsilon'}$, for $n\geq n(\epsilon')$ large enough.  Hence,
$$(II)^2 \leq C(\epsilon)^{3n\epsilon'} \E \left(\frac{||\bigwedge^2 S_{\tau(n)}||^{2\epsilon}}{||S_{\tau(n)}||^{4\epsilon}} \right)$$
Finally, by Corollary \ref{coro1}, there exists $\rho(\epsilon)\in ]0,1[$ such that for all $n$ large enough:
$$\E \left(\frac{||\bigwedge^2 S_{\tau(n)}||^{2\epsilon}}{||S_{\tau(n)}||^{4\epsilon}}\right) = \E \left( \big| \frac{a_2(\rho(S_{\tau(n)}))}{a_1
(\rho(S_{\tau(n)}))} \big|^{2\epsilon}\right) \leq \rho(\epsilon)^n$$
Choose $0<\epsilon ' < \frac{-\log(\rho(\epsilon))}{3\;\log(C(\epsilon))}$ so that for $\rho=C(\epsilon)^{3\epsilon'}\rho(\epsilon) \in ]0,1[$, $(II)^2 \leq \rho^n \leq ({\rho^{\frac{1}{2l}}})^N$. $$\Box$$

\section{Proof of Theorem \ref{main1}}
\label{secproof}
In this section, we complete the proof of Theorem  \ref{main1} and Corollary \ref{corexpo}.\\

Now let $\mu$ be a probability measure on $SL_d(k)$ such that $\Gamma_{\mu}$ is a strongly irreducible and contracting closed subgroup. We denote by $G$ the $k$-Zariski closure of $\Gamma_\mu$ in $SL_d(k)$, which we assume to be  $k$-split and its Zariski-connected component semi-simple.
We can apply the results of the previous Section \ref{subsequi} with this $G$ and $\rho$ the natural action of $G$ on $V=k^d$.
 We use the same notation and conventions as in Section \ref{generating}, regarding attracting points and repelling hyperplanes.

 We will show that \begin{equation}\limsup_{n \rightarrow \infty} {\frac{1}{n} \log \;\p \left(\langle S_n,S'_n\rangle  \textrm{do not form a ping-pong pair}\right) }<0\label{oulna}.\end{equation}  Applying lemma \ref{ping-pong}, this will end the proof of Theorem \ref{main1}.
 It will follow from the following two propositions.

\begin{prop} There exists $\epsilon \in ]0,1[$ such that for every $r\in ]\epsilon,1[$:
 $$\limsup_{n \rightarrow \infty}{\frac{1}{n} \log \; \p \left(S_n, S'_n \textrm{are not $(\epsilon^n,
 r^n)$- very proximal} \right) }<0$$ \label{bastanak}\end{prop}
\begin{prop} For every $t\in ]0,1[$; $$\limsup_{n \rightarrow \infty}{\frac{1}{n} \log \;\p \left(\delta(v_{{S_n}^{\pm 1}},H_{{S'_n}^{\pm 1}})\leq t^n \right)}\;< 0 $$
\label{fatetsenin}\end{prop}


\textbf{Proof of Proposition \ref{bastanak}}:
it will follow from Proposition \ref{contracfin} and Lemma \ref{klo}.
First, we recall Lemma \ref{cruciallemme} which says that a large ratio between the first two diagonal components in the
 $KAK$ decomposition  implies contraction. More precisely, let $\epsilon >0$. If $|\frac{a_2(g)}{a_1(g)}|\leq {\epsilon^2}$, then $[g]$ is $\epsilon$-contracting. Moreover, one can take $v_g = [k(g)e_1]$ to be the attracting point and $H_g$, the projective hyperplane spanned by $u^{-1}(g)e_i$ for $i=2,...,d$, to be the repelling hyperplane. \\
 We deduce the following proposition:

\begin{prop}\label{contracfin} There exists $\epsilon_0 \in ]0,1[$ such that for every $\epsilon \in ]\epsilon_0,1[$,

$$\limsup_{n \rightarrow \infty}{\frac{1}{n} \log \;\p \left(S_n \;\textrm{and $S'_n$}\;\textrm{are not $\epsilon^n$- very contracting} \right)}<0$$


\end{prop}
\begin{proof} It suffices to consider $S_n,S'_n$ and $S_n^{-1},{S'_n}^{-1}$ separately and show the corollary without the word ``very''. \\
$\bullet$ For the random walk $(S_n)$ Theorem \ref{ratiosld} shows that there exists $\epsilon_1\in ]0,1[$ such that for all large $n$ we have $ \E \left(\big|\frac{a_2(n)}{a_1(n)}\big|\right)\leq \epsilon_1^n $.  \;

By the Markov inequality, for every $\epsilon \in ]\epsilon_1,1[$, $$\p \left(\;\big|\frac{a_2(n)}{a_1(n)} \big|\;\geq \epsilon^n \right) \leq (\frac{\epsilon_1}{\epsilon})^n$$
 By Lemma \ref{cruciallemme}, for every $\epsilon\in ]\sqrt{\epsilon_1},1[$ we have
  $\p (S_n \textrm{\;is not $\epsilon^n$-  contracting}) \leq (\frac{\epsilon_1}{\epsilon^2})^n$.\\
$\bullet$ For the random walk $(S_n^{-1})$: The assumption
 $\Gamma_\mu$ is a group implies that
  $\Gamma_{\mu^{-1}}=\Gamma_\mu=\Gamma$ so that the action of $\Gamma_{\mu^{-1}}$ on $V$ is strongly irreducible and contracting. In consequence,
   we can apply the same reasoning as the previous paragraph by replacing $\mu$ with $\mu^{-1}$.
This gives $\epsilon_2\in ]0,1[$ such that for every $\epsilon \in ]\sqrt{\epsilon_2},1[$,  $\p (S_n^{-1} \textrm{\;is not $\epsilon^n$-  contracting})$ is sub-exponential.  \\
 Similarly if we denote by $\epsilon_3$, $\epsilon_4$ the quantities relative to $S'_n$ and ${S'_n}^{-1}$, then it suffices to choose $\epsilon_0=Max\{\sqrt{\epsilon_i}; i=1,...,4\}$

\end{proof}
Recall that for $g\in SL_d(k)$, $v_g=k(g)e_1$ and $H_g=\big[Span \langle u(g)^{-1}e_2,...,u(g)^{-1}e_d \rangle \big]$.

\begin{lemme} For every $t \in ]0,1[$, $$\limsup_{n \rightarrow \infty}{\frac{1}{n}\log\;\p \left(\delta(v_{S_n},H_{S_n}) \leq t^n \right)}<0$$
The same holds for ${S_n}^{-1}$, ${S'_n}$ and ${S'_n}^{-1}$.
\label{klo}\end{lemme}
\begin{proof} Consider the random walk $(S_n)_{n\in \N^*}$.
 Let $t \in ]0,1[$. Recall that if $H=Ker f$, $f\in V^*$ then for any non zero vector $x$ of $V$, $\delta([x],[H])= \frac{|f(x)|}{||f|| ||x||}$.
 Since $H_{S_n}=Ker (U_n^{-1}.e_1^*)$, we must show that for every $t\in ]0,1[$,
  \begin{equation}\label{hatem}\limsup_{n \rightarrow \infty} \frac{1}{n}\log\; \p (||U_n^{-1}.e_1^* (K_ne_1)  || \leq t^n)< 0\end{equation}

$\bullet$ For every $\epsilon>0$, let $\psi_\epsilon$ be  the  function defined on $\R$ by $\psi_\epsilon(x)= 1$ on $[-\epsilon,\epsilon]$; affine on $[-2\epsilon;-\epsilon[ \cup ]\epsilon,2\epsilon]$ and zero otherwise, for every $x\in \R$.\\
One can easily verify that $\psi_\epsilon$ is $\frac{1}{\epsilon}$-Lipschitz. \\
Note also that \begin{equation}\mathds{1}_{[-\epsilon,\epsilon]} \leq \psi_\epsilon \leq \mathds{1}_{[-2\epsilon,2\epsilon]}\label{yy}\end{equation}

$\bullet$ Let $\eta$ be the function on $P(V)\times P(V^*)$ defined by $\eta ([x],[f])=\delta\left([x],Ker(f) \right)=\frac{|f(x)|}{||f|| ||x||}$. \\
We consider the following metric on $P(V) \times P(V^*)$:  $d \left(([x],[f]), ([y],[g]) \right) = \delta([x],[y])+\delta([f],[g])$ for every $[x],[y]\in P(V)$ and $[f],[g]\in P(V^*)$. \\
Let $C(k)=\sqrt{2}$ when $k$ is  archimedean and $C(k)=1$ when $k$ is non archimedean. We claim that $\eta$ is $C(k)$-Lipschitz. Indeed, let $[x],[y]\in P(V)$, $[f],[g]\in P(V^*)$. By Lemma \ref{facile} below there exist suitable representatives  $x,y\in V$, $f,g\in V^*$  in the unit sphere such that $||x - y||\leq C(k) \delta([x],[y]) $ and $||f-g||\leq C(k) \delta([f],[g])$. But by the triangle inequality,
$\big |\eta ([x],[f]) - \eta ([y],[g])\big| \leq || f(x) - g(y)||\leq ||f -g|| + ||x - y||\leq C(k)
\left(\delta([f],[g])+\delta ([x],[y])\right)$.\\
Define for $\epsilon>0$, $\phi_\epsilon=\psi_\epsilon \circ \eta$. By the previous remarks, $\phi_\epsilon$ is $\frac{C(k)}{\epsilon}$- Lipschitz.\\
Theorem \ref{independence1} gives a $\rho \in ]0,1[$ and independent random variables $Z\in V$ and $T\in V^*$ such that for every Lipschitz  function $\phi$ on $P(V)\times P(V^*)$, and $n$ large enough
\begin{equation}\big|\E \left(\phi([K_ne_1],[U_n^{-1}.e_1^*])\right) - \E \left( \phi(Z,T)\right) \big|\leq ||\phi||\;\rho^n \label{imp} \end{equation}
where $||\phi||$ is the Lipschitz constant of $\phi$ as it was defined in Theorem \ref{independence1}.\\

 Now we prove (\ref{hatem}). For any $t\in ]0,1[$
\begin{eqnarray}\p (||U_n^{-1}.e_1^* (K_n e_1)|| \leq t^n ) & \leq & \E \left(\phi_{t^n} ([K_ne_1],[{U_n}^{-1}.e_1^*]) \right) \label{wiss}\\
& \leq & \E \left( \phi_{t^n} (Z,T) \right) + ||\phi_{t^n}||\;\rho^n \label{explique}\\
& \leq & \p (\frac{|T(Z)|}{||T|| ||Z||} \leq 2t^n ) + C(k)\frac{\rho^n}{t^n} \label{explique1}\\
&\leq & Sup\{\p \left(\delta (Z,[H]) \leq 2t^n \right);\;\textrm{$H$ hyperplane of $V$} \} + C(k)\frac{\rho^n}{t^n} \label{sallemouli}
\end{eqnarray}
The bound (\ref{explique}) follows from (\ref{imp}), while (\ref{wiss}) and (\ref{explique1}) use (\ref{yy}). Finally to get (\ref{sallemouli}) we used the independence of $Z$ and $T$. \\
By Theorem \ref{hausdweak},  (\ref{sallemouli}) is sub-exponential and the lemma is proved if $t>\rho$, a fortiori for every $t\in ]0,1[$.  $\Gamma_\mu$ being a group,  the action of $\Gamma_{\mu^{-1}}$ on $V$ is strongly irreducible and contracting, hence the same proof as above holds for $S_n^{-1}$. The roles of $S_n$ and $S'_n$ are interchangeable.
\end{proof}

\begin{lemme}\label{facile} Let $C(k)=\sqrt{2}$ when $k$ is archimedean and $C(k)=1$ when $k$ is not. Then for any $[x],[y]\in P(V)$, there exist representatives in the unit sphere such that
$$\delta([x],[y])\leq || x -y|| \leq C(k) \delta ([x],[y])$$
(In particular, in the non archimedean case these are equalities). The same holds for $V^*$.\end{lemme}

\begin{proof} Let $x$ and $y$ be representatives of norm one of $[x]$ and $[y]$.
When $k=\C$, denote by $<.,.>$ the canonical scalar product on $k^d$. Then  $\delta([x],[y])^2 = 1 - |<x,y>|^2=  \left(1-Re(<x,y>)\right)\left(1+Re(<x,y>) \right)$. One can choose $x$ and $y$ in such a way that $<x,y>\in \R$ and $Re(<x,y>)\geq 0$. The identity  $||x - y||^2= 2\left(1-Re(<x,y>) \right)$ ends the proof.
The case $k=\R$ is similar. \\
When $k$ is non archimedean, recall that by definition:
$\delta([x],[y]) = Max \{|x_i y_j - x_j y_i|; i \neq j \} \label{kk}$.
The norm being ultrametric,  for any $i,j$,  $|x_iy_j - x_j y_i|= |y_j (x_i-y_i) +y_i(y_j-x_j)|\leq ||x - y||$. Hence $\delta([x],[y])\leq ||x - y||$. For the other inequality, we distinguish two cases:\\
$\bullet$ Suppose that there is an index $m$ such that $x_m$ and $y_m$ are of norm one (i.e. in $\Omega_k^*$). By rescaling if necessary $x$ and $y$, one can suppose that $x_{m}=y_{m}=1$.
Without loss of generality we can assume
that $m=1$.  Hence, $\delta([x],[y])\geq Max\{|x_i -y_i|;\;i\geq 2\}= ||x - y||$.\\
$\bullet$ Suppose that there is no index $m$ such that $x_m$ and $y_m$ are of norm one. Let $i_0$ (resp. $j_0$) be an index such that $x_{i_0}$ (resp. $y_{j_0}$ ) is invertible: such indices exist because $x$ and $y$ are on the unit sphere. $i_0 \neq j_0$ and neither $x_{j_0}$ nor $y_{i_0}$ is of norm one.  Hence, $|x_{i_0} y_{j_0} - y_{i_0} x_{j_0}| =  1$ and $\delta([x],[y])= 1 = ||x - y||$.

\end{proof}

\textbf{Proof of Proposition \ref{fatetsenin}:}

Let $t>0$. On the one hand for every given $n$  $S_n$ and $M_n$ have the same law and on the other hand  $(X_1,...,X_n)$ and $(X'_1,...,X'_n)$ are independent, hence
\begin{eqnarray}
\p \left(\delta(v_{S_n},H_{{S'_n}^{\pm 1}})\leq t^n \right)&=& \p \left(\delta\left(k(M_n)[e_1],H_{{S'_n}^{\pm 1}}\right)\leq t^n\right)\label{etapes} \\
&\leq & Sup\{\p \left(\delta\left(k(M_n)[e_1],H\right)\leq t^n \right); \;\textrm{$H$ hyperplane of $V$}\}\label{kkl}\end{eqnarray}
By Theorem \ref{direction1} and the Markov inequality, there exist $\rho_1,\rho_2\in ]0,1[$, a random variable $Z$ in $P(V)$ such that:
 \begin{equation}\label{markov}\p \left(\delta (k(M_n)[e_1],Z)\geq \rho_1^n \right) \leq \rho_2^n \end{equation}
(\ref{kkl}), (\ref{markov}) and the triangle inequality give:
$$\p \left(\delta([v_{S_n}],[H_{S'_n}])\leq t^n \right) \leq Sup\{ \p \left( \delta(Z,[H])\leq t^n +\rho_1^n \right);\,\textrm{$H$ hyperplane of $V$}\} + \rho_2^n $$
Theorem \ref{hausdweak} shows that the latter is exponentially small.  We may of course exchange the roles of $S_n$ and $S'_n$. When we  consider $S_n^{-1}$ instead of $S_n$ the same estimates hold. Indeed, as explained in the proof of Proposition \ref{contracfin}, $\Gamma_{\mu^{-1}}$ acts strongly irreducibly on $V$ and contains a contracting sequence.
\begin{flushright} $\Box$ \end{flushright}

 \textbf{Proof of Corollary \ref{corexpo}}: let $l\in \N^*$ and $(M_{n,1})_{n\in \N^*}$,...,$(M_{n,l})_{n\in \N^*}$ be $l$ independent random walks associated to $\mu$.
Propositions \ref{bastanak} and \ref{fatetsenin} give $\epsilon,r,\rho\in ]0,1[$, $n_0\in \N^*$ such that for every $n>n_0$ and $i,j\in \{1,...,l\}$, $\p (A_{n,i,j})\leq \rho^n$ and $\p ({B_{n,i,j}})\leq \rho^n$, where $A_{n,i,j}$ is the event `$`M_{n,i}$ and $M_{n,j}$ are not $(r^n,\epsilon^n)$-very proximal'' and $B_{n,i,j}$ is the union of the $4$ events: the attracting point of $M_{n,i}^{\pm 1}$ is at most $\epsilon^n$-apart from the repelling hyperplane of $M_{n,j}^{\pm 1}$. Hence for every $l\in \N^*$ and $n>n_0$: $$\p (\textrm{$M_{n,1}$,...,$M_{n,l}$ do not form a ping-pong
$l$-tuple}) \leq \sum_{i<j}{\p(A_{i,j})+\p(B_{i,j})}\leq l(l-1)\rho^n$$ Fix $n>n_0$ and let $\rho'\in ]\rho,1[$, $l_n=\lfloor\frac{1}{\rho'^n}\rfloor$. The previous estimate shows that if $(M_{k,1})_{k\in \N^*}$,...,$(M_{k,l_n})_{k\in\N^*}$ are $l_n$ independent and identically distributed random walks, then the probability\\ $\p (\textrm{$M_{n,1}$,...,$M_{n,l_n}$ do not form a ping-pong $l_n$-tuple})$ decreases exponentially fast.

\begin{flushright} $\Box$ \end{flushright}

\end{document}